\documentclass[%
 aip,
 amsmath,amssymb,
 reprint,%
]{revtex4-1}

\usepackage{graphicx}
\usepackage{dcolumn}
\usepackage{bm}

\usepackage[utf8]{inputenc}
\usepackage[T1]{fontenc}
\usepackage{mathptmx}
\usepackage{etoolbox}
\usepackage{rotating}
\usepackage[outdir=./]{epstopdf}

\makeatletter
\def\@email#1#2{%
 \endgroup
 \patchcmd{\titleblock@produce}
  {\frontmatter@RRAPformat}
  {\frontmatter@RRAPformat{\produce@RRAP{*#1\href{mailto:#2}{#2}}}\frontmatter@RRAPformat}
  {}{}
}%
\makeatother

\usepackage{amsmath,amsfonts,amssymb,amsthm,dsfont, bm}
\usepackage{array}
\usepackage{enumitem}

\usepackage{float}
\usepackage[usenames,dvipsnames]{xcolor}
\usepackage[colorlinks=true]{hyperref} 
\usepackage{verbatim}
\usepackage{overpic}
\usepackage{scalerel}
\usepackage{booktabs}

\hypersetup{linkcolor=Violet,citecolor=PineGreen}
\newtheorem{thm}{Theorem}[section]

\newtheorem{prop}[thm]{Proposition}

\theoremstyle{definition}
\newtheorem{defn}[thm]{Definition}

\newtheorem{rmk}[thm]{Remark}
\numberwithin{equation}{section}

\newenvironment{sistema}%
{\left\lbrace\begin{array}{@{}l@{}}}%
{\end{array}\right.}

\newcommand{\Hu}{\mathcal{H}}
\newcommand{\N}{\mathbb{N}}
\newcommand{\R}{\mathbb{R}}

\newcommand{\ep}{\varepsilon}

\newcommand{\PreserveBackslash}[1]{\let\temp=\\#1\let\\=\temp}
\newcolumntype{C}[1]{>{\PreserveBackslash\centering}p{#1}}
\newcolumntype{R}[1]{>{\PreserveBackslash\raggedleft}p{#1}}
\newcolumntype{L}[1]{>{\PreserveBackslash\raggedright}p{#1}}

\newcommand{\imag}{\textbf{i}}

\graphicspath{{images/}{../images/}}

\makeatletter
\@namedef{subjclassname@2020}{%
  \textup{2020} Mathematics Subject Classification}
\makeatother

\usepackage{subcaption}

\begin{document}

\preprint{AIP/123-QED}

\title[On the transition between autonomous and nonautonomous systems: the case of FitzHugh-Nagumo's model]{On the transition between autonomous and nonautonomous systems: the case of FitzHugh-Nagumo's model}
\author{I.P.~Longo}
\email{iacopo.longo@imperial.ac.uk.}
 \affiliation{Imperial College London,
Department of Mathematics, 635 Huxley Building, 180 Queen’s Gate
South Kensington Campus, SW7 2AZ London, United Kingdom.}
\author{E.~Queirolo}
\affiliation{Technical University of Munich, School of Computation Information and Technology, Department of Mathematics,
Boltzmannstraße 3,
85748 Garching bei M\"unchen, Germany.}
\author{C.~Kuehn}
 \affiliation{Technical University of Munich, School of Computation Information and Technology, Department of Mathematics,
Boltzmannstraße 3,
85748 Garching bei M\"unchen, Germany.}

\date{\today}

\begin{abstract}
This work deals with a parametric linear interpolation between an autonomous  FitzHugh-Nagumo model and a nonautonomous skewed-problem with the same fundamental structure. This paradigmatic example allows to construct a family of nonautonomous dynamical systems with an attracting integral manifold and a hyperbolic repelling trajectory located within the nonautonomous set enclosed by the integral manifold. 
Upon the variation of the parameter the integral manifold collapses, the hyperbolic repelling solution disappears and a unique globally attracting hyperbolic solution arises in what could be considered yet another nonautonomous Hopf bifurcation pattern.
Interestingly, the three phenomena do not happen at the same critical value of the parameter, yielding thus an example of a nonautonomous bifurcation in two steps. 
We provide a mathematical description of the dynamical objects at play and analyze the periodically forced case via rigorously validated continuation.
\end{abstract}

\maketitle

\begin{quotation}
We investigate the transition between autonomous and nonautonomous systems using the FitzHugh-Nagumo model as a paradigmatic example. We analyze the structure of the global attractor of the system while a parameter varies, and unveil a nonautonomous bifurcation in which an attracting set collapses, a repelling solution disappears and a unique globally attracting solution arises. We provide rigorous mathematical justification of the involved dynamical objects, and validated numerics of the bifurcation via continuation techniques in the periodic case.
\end{quotation}

\section{Introduction}\label{secintro}

Multiple time scale~\cite{} and non-autonomous~\cite{} nonlinear dynamics have been separated in many instances. However, these classes share many physical features, e.g., a slowly drifting variable with constant speed directly leads to a non-autonomous system. Furthermore, a system with time-dependent forcing can be made formally autonomous by augmenting time as a phase space variable. If the forcing is very fast or very slow, a natural multiple time scale structure emerges. This leads to the natural approach that one might want to interpolate between these two system classes in the hope of relating results from their respective analysis. One goal of this work is to follow this approach and uncover its advantages but also its substantial challenges. We are going to focus on one of the simplest paradigmatic nonlinear dynamics models and consider a parametric linear interpolation between a FitzHugh-Nagumo model (including the classic van der Pol system) and a skewed nonautonomous differential problem sharing the same fundamental structure. Specifically, we consider the following nonautonomous planar system of  ordinary differential equations,
\begin{equation}\label{eq:problem}
\begin{sistema}
x' = (1-\delta)y+\delta v(t)-\frac{x^3}{3} +x,\\
y' =\ep(a-x-b y),\\
v\in L^\infty(\R,\R)
\end{sistema}
\end{equation}
with $0\le \delta\le 1,\, \ep,b\ge0,\, a\in\R$.
The first immediate observations concern the family of dynamical systems~\eqref{eq:problem} as $\delta$ varies in $[0,1]$. If $\delta=0$,~\eqref{eq:problem} is the classic autonomous FitzHugh-Nagumo model with constant forcing $a$,
\begin{equation}\label{eq:FHN}
\begin{sistema}
x' = y-\frac{x^3}{3} +x,\\
y' =\ep(a-x-by).\\
\end{sistema}
\end{equation}
Note that if $b=0$ then \eqref{eq:FHN} reduces to the van der Pol model~\cite{van1927vii}, which is an example of an oscillator with nonlinear damping, energy being dissipated at large amplitudes and generated at low  amplitudes. 
Typically, this feature guarantees the existence of a limit cycle in a suitable region of the parameter space; that is, sustained oscillations around a state at which energy generation and dissipation balance.
Van der Pol's equation is a classic topic of nonlinear dynamics and it has been extensively studied (see for example Ref.~\onlinecite{kuehn2015multiple,guckenheimer2013nonlinear} and references therein). 
The FitzHugh-Nagumo model \cite{fitzhugh1955mathematical} generalizes the van der Pol model, enabling the appearance of a larger variety of oscillations and modelling several processes in biophysics in a bit more detail~\cite{}, such as excitable neuron dynamics. 
Unsurprisingly, the bifurcation analysis of the FitzHugh-Nagumo model is even richer and more involved. We point the interested reader towards the detailed presentation in Ref.~\onlinecite{rocsoreanu2012fitzhugh} for further information. 
Interestingly, the FitzHugh-Nagumo model is known to showcase dynamical complexity when subjected to nonautonomous forcing\cite{kuehn2017quenched}. More in general, the appearance of almost automorphic dynamics and strange nonchaotic attractors in almost periodically forced, damped nonlinear oscillators is a classic field of study \cite{huang2009almost, yi2004almost,romeiras1987strange,pikovsky2006strange, jager2006denjoy,jager2006towards}.

Important for our work is the existence of an attractive limit cycle for certain values of the parameters of the autonomous model.  This is guaranteed by the occurrence of a supercritical Hopf bifurcation where a stable equilibrium becomes unstable and the attractive periodic orbit appears. 

For $\delta=1$,~\eqref{eq:problem} is the skewed nonautonomous problem
\begin{equation}\label{eq:NA_limit}
\begin{sistema}
x' = v(t)-\frac{x^3}{3} +x,\\
y' =\ep(a-x-by).\\
\end{sistema}
\end{equation}
By the term ``skewed" we mean that~\eqref{eq:NA_limit} is composed of an uncoupled cubic nonautonomous equation and by a further scalar linear problem in $y$ that can be easily integrated with the variation of constants formula once the solution of the first equation is known. The dynamical scenarios of the first equation in~\eqref{eq:NA_limit} have been studied by Tineo \cite{tineo2003result,tineo2003result2}, and more recently also by Due\~nas et al.~\cite{duenas2022generalized} under far more general assumptions. 
In particular, for any $v$ bounded and uniformly continuous, there are at least one and at most three hyperbolic solutions of 
\begin{equation*}
\dot x = v(t)-\frac{x^3}{3} +x, 
\end{equation*}
i.e., trajectories that are hyperbolic in the non-autonomous sense as explained in Section \ref{subsec:attractors}. Additionally, if three hyperbolic solutions exist, then the upper and the lower ones are attractive and the one in between is repelling, whereas if only one hyperbolic solution exists, then it is always attractive. 
We will work under the assumption that this equation has indeed only one hyperbolic solution and therefore it is attracting. 
Additionally, if $v$ is assumed periodic, then we show that an attractive periodic solution exists for~\eqref{eq:NA_limit}.

If $0<\delta<1$, the system~\eqref{eq:problem} can be understood \mbox{(bio-)physically} as a FitzHugh-Nagumo type of model
with time-dependent injected current. More theoretically,~\eqref{eq:problem} corresponds to an interpolation between two different FitzHugh-Nagumo models via the parameter $\delta$. The two limits $\delta=0$ and $\delta=1$ correspond to a classical multiple time scale and to a non-autonomous systems respectively.

The variation of the parameter $\delta$ connects dynamical systems living on different phase spaces: the autonomous problem for $\delta=0$ has phase-space $\R^2$ whereas the nonautonomous problems for $\delta>0$ are defined on the extended phase-space $\R\times\R^2$.
In fact, for $\delta>0$, standard techniques from topological dynamics allow to construct an autonomous flow on an extended phase space, seen as a bundle of a base and fibers, which is dynamically much more interesting. The base corresponds to a functional metric space on which the time shift (Bebutov flow) produces a well-defined dynamical system. Under certain assumptions on $v$, the base of this bundle is in fact compact, unlike the real line of the trivial augmented phase space $\R\times\R^2$. The fiber, instead, is the original Euclidean space now parametrized over the elements of the base. The resulting flow is called a skew-product flow. A brief summary of these facts is presented in \ref{subsec:skew-prod}.
As an example, if $v$ is periodic, a standard trick allows to reduce the nonautonomous problems to autonomous ones on the extended phase space $\mathbb{S}^1\times\R^2$---the variable of time can be considered modulo the period of $v(t)$ which is fixed a priori. 
In order to provide a unified framework, also the phase space of \eqref{eq:FHN} is accordingly augmented. 
This fact, however, has important consequences on the interpretation of the dynamics of \eqref{eq:FHN}. For example, in the simplest nontrivial case with $v$ nonconstant periodic, an equilibrium point and a periodic orbit of the autonomous FitzHugh-Nagumo (or the van der Pol) model will respectively correspond to a periodic orbit and to the boundary of a two-torus in the augmented phase-space $\mathbb{S}^1\times\R^2$.

Consequently, even in the case where $y$ is substituted by the $y$-coordinate of a twin independent Fitzhugh-Nagumo model, 
\[
\begin{split}
&\begin{sistema}
x' = v(t)-\frac{x^3}{3} +x\\
y' =\ep(a-x-by),\\
\end{sistema}\\
&\text{where } \big(u(t),v(t)\big)\text{ solves}\quad
\begin{sistema}
u' = v-\frac{u^3}{3} +u\\
v' =\ep(a-u-bv),
\end{sistema}\\
\end{split}
\]
the dynamical scenarios of \eqref{eq:FHN} and \eqref{eq:NA_limit} will be fundamentally different. The global attractor of the first one being a two-torus, while for the second one a single attracting periodic orbit. 

We investigate the change of the global attractor and specifically the occurrence of bifurcations upon the variation of the parameter $\delta\in[0,1]$. In the periodic case, this also corresponds to the breaking of the torus. The description of the global attractor is carried out by combining analytical techniques and rigorous validated numerics. The former allow a precise account of the fine structure when $\delta$ assumes values close to $0$ and $1$. The latter shed light on the bifurcation. 
The pattern we encounter can be considered yet another version of nonautonomous Hopf bifurcation\cite{braaksma1987quasi,braaksma1990toward,johnson1994hopf,Franca2016nonautonomous,anagnostopoulou2015model,franca2019non,nunez2019li} in the sense that the global attractor with positive measure in the fiber existing close to $\delta=0$ discontinuously collapses giving birth to a unique attracting hyperbolic solution. 
In the process, a hyperbolic completely repelling solution, contained in the original global attractor, also disappears. Unlike an autonomous Hopf bifurcation, the diameter of the global attractor does not shrink close to the bifurcation point but it abruptly vanishes, something not uncommon in the nonautonomous setting\cite{nunez2019li}. Ref.~\onlinecite{Franca2016nonautonomous} contains, to the best of our knowledge, the most similar setup, in that their object of study is a nonautonomous perturbation of the normal form of the autonomous Hopf bifurcation. We, however, do not assume to work in a range of parameters close to an autonomous Hopf bifurcation but rather take the existence of a periodic limit cycle for the autonomous problem as a starting point and vary a parameter which gradually leads to a nonautonomous problem which is not a perturbation of the autonomous one.

The work is organized as follows. In Section \ref{sec:topological dynamics}, we present some background notions on nonautonomous dynamical systems such as the skew-product formalism and the process formalism, the notions of attractors for nonautonomous systems, hyperbolic solutions, and integral manifolds.  This section might seem a little technical. However, as we aim to link non-autonomous and autonomous multiple time scale systems, we first have to carefully  set up the background for both areas. 

Section \ref{sec:global-attractor} deals with the description of the global attractor of \eqref{eq:problem} whose existence  for all $\delta\in[0,1]$ is proved in \ref{subsec:existence-glob-attr}; we distinguish three cases: in  \ref{subsec:global-attractor-d=0} we deal with the case $\delta=0$ employing classical techniques from multiple time scale dynamics\cite{krupa2001relaxation} to recall the existence of a strongly attracting periodic orbit which, in turn, induces a trivial attracting integral manifold in the augmented phase space. Its persistence under small perturbations is used to inherit the existence of a nontrivial integral manifold for strictly positive values of $\delta$ close to zero. 
In \ref{subsec:global-attract-d=1} we apply recent results on the characterization of hyperbolic solutions for scalar d-concave ordinary differential equations\cite{duenas2022generalized} to prove the existence of a unique attracting hyperbolic solution for \eqref{eq:NA_limit}. Its persistence under small perturbations is used to guarantee the same dynamical scenario for values of $\delta$ close to one. In \ref{subsec:global-attractor-int-d}, we give a rigorous description of the finite-time behaviour of the attractor for intermediate values of $\delta$ using recent results in singularly perturbed nonautonomous systems\cite{longo2024tracking}.

Section \ref{sec:nonauton-bif} tackles the analysis of the nonautonomous bifurcation leading to the collapse of the integral manifold, the disappearance of the repelling hyperbolic solution and the appearance of the globally attracting hyperbolic solution. After providing numerical evidence of the two-step bifurcation for the periodically and quasi-periodically forced cases, we focus on a four dimensional autonomus system where the forcing is generated by a twined FitzHugh-Nagumo system whose $y$-component is fed into the two-dimensional parametric problem as $v$. In this setup we are able to employ a numerically validated approach to carry out the continuation of the hyperbolic periodic trajectories of the two-dimensional forced problem and study their stability by approximating the Lyapunov spectra for the variational equations along them.


\section{Basic facts on flows, attractors, hyperbolic solutions and integral manifolds}\label{sec:topological dynamics}

\subsection{Skew-product flows and processes}\label{subsec:skew-prod} Let $v:\R\to\R$ be a bounded and uniformly continuous function. A standard technique allows to construct an autonomous flow for \eqref{eq:problem} on a compact metric space. 
Let $\mathcal{H}(v)$ denote the closure with respect to the compact-open topology of the set of time-translations of $v$ i.e., the set $\mathrm{cls}\{v_t:\R\to\R\mid t\in\R\}$, where $v_t(s)=v(t+s)$. 
It is well-known that $\mathcal{H} (v)$ is a compact metric space and the shift map $\sigma:\R\times \mathcal{H}(v)\to\mathcal{H}(v)$, $(t,v)\mapsto \sigma_t(v)=v_t$ defines a continuous flow on $\mathcal{H}(v)$ called the Bebutov flow. 
This flow is also minimal  provided that  $v$ has certain recurrent behaviour in time, such as periodicity, almost periodicity, or other weaker properties of recurrence. If $v$ is almost periodic, then the flow on the hull is minimal and almost periodic, and thus uniquely ergodic. Let us now consider $w\in\mathcal{H}(v)$, $X_0=(x_0,y_0)\in\R^2$, and denote by $X(t,w,X_0)$ the maximal solution of the Cauchy problem 
\begin{equation}\label{eq:SkewProd-CP}
\begin{sistema}
x' = (1-\delta)y+\delta w(t)-\frac{x^3}{3} +x\\
y' =\ep(a-x-b y)\\
x(0)=x_0,\ y(0)=y_0
\end{sistema}
\end{equation}
with $0\le \delta\le 1,\, \ep,b\ge0,\, a\in\R.$
The map 
\[
\begin{split}
 \Phi:\R\times\mathcal{H}(v)\times \R^2&\to\mathcal{H}(v)\times \R^2\\
 (t,w,X_0)&\mapsto\Phi(t,w,X_0)=(w_t,X(t,w,X_0))   
\end{split}
\]
is a continuous skew-product (autonomous) flow on $\mathcal{H}(v)\times \R^2$ on the compact metric base $\mathcal{H}(v)$. In the special case that $v$ is periodic, the identification $\mathcal{H}(v)=\mathbb{S}^1$ holds.
If $v$ is quasi-periodic, $\mathcal{H}(v)=\mathbb{T}^n$ for some $n\in\N$, but already for $v$ almost periodic, $\Hu(v)$ is not a manifold in general.
We note that, in fact, the construction of a continuous skew-product flow for \eqref{eq:problem} holds under substantially weaker assumptions akin to local integrability (see Ref.~\onlinecite{longo2018topologies} and references therein). 
For example, \eqref{eq:problem} induces a continuous skew-product flow also if $v$ is essentially bounded, i.e.~$v\in L^\infty(\R,\R)$ where an integral topology is used in place of the compact-open topology (see Theorem 5.6(i) in Ref.~\onlinecite{longo2017topologies}). In this case, a solution is intended in the sense of Carathéodory, that is, a locally absolutely continuous function that solves the associated integral problem. Although the compactness of the base does not hold in general under such mild assumptions of regularity, it does hold in the specific case of \eqref{eq:problem} (for general characterizations of compacteness with respect to strong and weak integral topologies see Ref.~\onlinecite{artstein1977topological,george1968compact}).

\begin{prop}\label{prop:SKEW-PROD-COMPACT}
    Let $v\in L^\infty(\R,\R)$. Then, $H(v)$ is a compact metric space and \eqref{eq:problem} induces a continuous skew-product flow on a compact metric base
    \[\begin{split}
    \Phi:\R\times\mathcal{H}(v)\times \R^2&\to\mathcal{H}(v)\times \R^2,\\ (t,w,X_0)&\mapsto\Phi(t,w,X_0)=(w_t,X(t,w,X_0)).
    \end{split}
    \]
\end{prop}
\begin{proof}
    Let $f(t,X)$ denote the right-hand side of \eqref{eq:problem} with $X=(x,y)$, and let $J_Xf(t,X)$ denote the Jacobian of $f$ with respect to the variable $X$. Additionally, let $\|v\|_\infty=V$. For any $r>0$, there are positive constants $m=m(V,r)$ and $l=l(r)$ such that for almost every $t\in\R$, $|f(t,X)|\le m$ and 
    \[\begin{split}
    &|f(t,X_1)-f(t,X_2)|\\
    &\le  |X_1-X_2|
    \int_0^1\left|
    J_Xf\big(t,X_1+\theta(X_2-X_1)\big)
    \right|
    d\theta \le l\,|X_1-X_2|,
     \end{split}
    \]
    for every $X,X_1,X_2\in B_r(0)$.
    Therefore, in the terminology of Carath\'eodory functions, $f$ has $L^1_{loc}$-uniformly integrable $m$-bounds and $L^1_{loc}$-bounded $l$-bounds---in fact, essentially bounded $m$-bounds and $l$-bounds. Therefore, Theorem 4.1 in Ref.~\onlinecite{artstein1977topological} guarantees that $H(v)$ is a compact metric space, whereas Theorem 3.5(i) and Theorem 2.33 in Ref.~\onlinecite{longo2018topologies} show that the induced skew-product flow is continuous.
\end{proof}


Instead of considering a flow on an extended phase space, sometimes it is useful to look at a nonautonomous problem in terms of an evolution operator which now depends on two parameters, the initial and the final time. Such evolution operator goes under the name of process or cocycle.
\begin{defn}
Given a metric space $(E,d)$ a process is a family of continuous maps $\{S(t,s)\mid t\ge s\}\subset \mathcal{C}(E)$ satisfying
\begin{itemize}[leftmargin=*]
\item $S(t,t)\,x=x$ for every $t\in\R$ and $x\in E$.
\item $S(t,s)=S(t,r)\,S(r,s)$, for every $t\ge r \ge s$.
\item $(t,s,x)\mapsto S(t,s)\,x$ is continuous for every $t\ge s$ and $x\in E$.
\end{itemize}
\end{defn}
\begin{rmk}
 If $\dot X=f(t,X)$ is such that for any $X_0\in\R^N$ and any $t_0\in\R$, there is a unique solution $X(\cdot, t_0,X_0)$ and it is defined on $[t_0,\infty)$, then a process is induced by
\begin{equation}\label{eq:process}
S_f(t+t_0,t_0)\,X_0=X(t+t_0,t_0,X_0),
\end{equation}
where $t\ge 0$ and $t_0\in \R$.  
\end{rmk}
\subsection{Attractors for nonautonomous differential equations}\label{subsec:attractors}
The double dependence on time of nonautonomous problems gives rise to a new type of attractivity---pullback attractivity---which is independent from the classic notion of forward attractivity. The two notions are indistinguishable for autonomous systems. 
Next we recall the definitions of global and pullback attractor for skew-product flows as well as the definition of pullback attractor for a process on $\R^2$. For an extensive treatment on the matter, we point the reader to Ref.~\onlinecite{kloeden2011nonautonomous,carvalho2012attractors} and the references therein.

\begin{defn}[Pullback attractor for a process]
A family of subsets $\mathcal A(\cdot)=\{\mathcal A(t) \mid t \in \R\}$ of the phase space $\R^2$  is said to be a \emph{pullback attractor for the process $S(\cdot,\cdot)$} if
\begin{itemize}
\item[(i)] $\mathcal A(t)$ is compact for each $t\in\R$;
\item[(ii)] $\mathcal A(\cdot)$ is invariant, that is, $S(t, s)\,\mathcal A(s) = \mathcal A(t)$ for all $t\ge  s$;
\item[(iii)] for each $t\in \R$, $\mathcal A(t)$ pullback attracts bounded sets at time t, i.e. for any bounded set $B\subset \R^2$ one has
\[
\lim_{s\to-\infty} d(S(t, s)\,B,\mathcal A(t)) = 0,
\] 
where $d(A,B)$ is the \emph{Hausdorff semi-distance} between two nonempty sets $A,B\subset \R^2$  i.e. $d(A,B) := \sup_{x\in A} \inf_{y\in B}d(x,y)$.
\item[(iv)] $\mathcal A$ is the minimal family of closed sets with property (iii).
\end{itemize}
The pullback attractor is said to be  \emph{bounded} if  $\bigcup_{t\in\R}\mathcal A(t)$ is bounded.
\end{defn}

\begin{defn}[Pullback and global attractors for a skew product flow]
Assume that for any $w\in\mathcal{H}(v)$ and any $X_0\in \R^2$, the solution $X(\cdot,w,X_0)$ of \eqref{eq:SkewProd-CP} is defined on $[0,\infty)$, i.e. the induced  skew-product flow  is defined on $\R^+\!\!\times\mathcal{H}(v)\times \R^2$. (This fact will be proved true in Theorem \ref{prop:globAttr}).\smallskip
\begin{itemize}[leftmargin=12pt]
\item A family $\widehat A=\{ A_w\subset\R^2\mid w\in\mathcal{H}(v)\}$ of nonempty, compact sets of $\R^2$ is said to be a \emph{pullback attractor for the skew-product flow} if it is invariant, i.e.
\begin{equation}\label{eq:13.06-13:32}
X(t,w,A_w)=A_{w_t} \quad \text{ for each } t\ge 0 \;\text{ and }\; w\in\mathcal{H}(v)\,,
\end{equation}
and, for every nonempty bounded set $D$ of $\R^2$ and every $w\in\mathcal{H}(v)$ one has
\begin{equation}\label{eq:13.06-12:33}
\lim_{t\to\infty} d(x(t,w_{-t},D), A_w)=0\,,
\end{equation}
where $d(A,B)$ denotes the \emph{Hausdorff semi-distance} of two nonempty sets $A$, $B$ of~$\R^2$.

A pullback attractor for the skew-product flow is said to be bounded if
\begin{equation*}
\bigcup_{w\in\mathcal{H}(v)} A_w\quad\text{is bounded.}
\end{equation*}
\item A compact set $\mathcal A$ of $\mathcal{H}(v)\times \R^2$ is said to be a \emph{global attractor for the skew-product flow} if it is the maximal nonempty compact subset of $\mathcal{H}(v)\times \R^2$ which is $\Phi$-invariant, i.e.
\begin{equation*}
 \Phi(t,\mathcal A)=\mathcal A \quad \text{ for each }\; t\ge 0\,,
\end{equation*}
and attracts all compact subsets $\mathcal D$ of $\mathcal{H}(v)\times \R^2$, i.e.
\[\lim_{t\to\infty} d(\Phi(t,\mathcal D),\mathcal A)=0\,,\]
where now $d(\mathcal B,\mathcal C)$ denotes the \emph{Hausdorff semi-distance} of two nonempty sets $\mathcal B$, $\mathcal C$ of $\mathcal{H}(v)\times \R^2$.
\end{itemize}
\end{defn}
\begin{rmk}\label{rmk:pullbackFROMskewTOPROCESS}
We wish to make explicit the relation between the family of sets $A_w\subset\R^2$, with $w\in\mathcal{H}(v)$, in the previous definition of the pullback attractor of the skew-product flow $\widehat A=\{ A_w\subset\R^2\mid w\in\mathcal{H}(v)\}$ and the pullback attractors of the processes induced by each $w\in\mathcal{H}(v)$.
Recalling that for  every $w\in\mathcal{H}(v)$ one has
\begin{equation}\label{eq:13.06-12:51}
S_w(s+r,r)\,x_0=x(s,w_r,x_0)\,,
\end{equation}
then, for every nonempty bounded set $D$ of $\R^2$, \eqref{eq:13.06-12:33} becomes
\begin{equation}\label{eq:13.06-12:46}
\lim_{t\to\infty} d(S_w(0,-t)\,D, A_w)=0\,,
\end{equation}
which implies that, for the given process $S_w(\cdot,\cdot)$, $ A_w$ pullback attracts bounded sets at time $0$. In turn, for every $\tau\in\R$ one can write~\eqref{eq:13.06-12:46} for $w_\tau$ and $A_{w_\tau}$ instead of $w$ and $A_w$, respectively, i.e.
\begin{equation}\label{eq:13.06-12:54}
\lim_{t\to\infty} d(S_{w_\tau}(0,-t)\,D, A_{w_\tau})=0\,.
\end{equation}
Nevertheless, using \eqref{eq:13.06-12:51} twice, we also have
\begin{equation*}
S_{w_\tau}(0,-t)\,D=x(t,w_{\tau-t},D)=S_w(\tau,\tau-t)\,D\,,
\end{equation*}
and thus \eqref{eq:13.06-12:54} can be written as
\begin{equation}\label{eq:13.06-13:35}
\lim_{t\to\infty} d(S_w(\tau,\tau-t)\,D, A_{w_\tau})=0\,,
\end{equation}
which implies that, for the given process $S_w(\cdot,\cdot)$, $ A_{w_\tau}$ pullback attracts bounded sets at time $\tau$. Therefore, as a consequence of the invariance contained in \eqref{eq:13.06-13:32} and of the fact that $A_w$ is taken compact for any $w\in\mathcal{H}(v)$, we deduce that the process $S_w(\cdot,\cdot)$ has a pullback attractor. In particular if
\begin{equation*}
\mathcal{A}_w=\left\{A(\tau)=A_{w_\tau}\mid\tau\in\R\right\}
\end{equation*}
is the minimal family of closed sets satisfying \eqref{eq:13.06-13:35}, then $\mathcal{A}_w$ is the pullback attractor for the process $S_w(\cdot,\cdot)$. On the other hand if for any $w\in\mathcal{H}(v)$, the induced process   $S_w(\cdot,\cdot)$ has a pullback attractor and  $A_w$ denotes the section at time $0$ of the pullback attractor of $S_w(\cdot,\cdot)$, then one has that
\begin{equation*}
\widehat A=\left\{ A_w\mid w\in\mathcal{H}(v)\right\}
\end{equation*}
is a pullback attractor for the skew product flow on $\mathcal{H}(v)\times \R^2$.\par\smallskip

\end{rmk}

A special case of nonautonomous (local) attractors is given by hyperbolic attracting solutions, i.e.~globally defined solutions which uniformly pullback and forward attract nearby solutions at an exponential rate. Let us be more precise. Let $L:\R\to\R^2$ be a continuous function, and recall that a linear homogeneous system \begin{equation}\label{eq:linearNAODE}
\dot z=L(t)z,    
\end{equation} 
is uniformly asymptotically stable if and only if the trivial solution $x\equiv0$ is exponentially stable (see for example Theorem~4.4.2 in Ref.~\onlinecite{adrianova1995introduction}), i.e.~there are constants $K>0$ and $\gamma>0$ such that,
\begin{equation}\label{eq:exp-stable-linear}
\|\Psi(t,s)\|\le Ke^{-\gamma(t-s)},\quad\text{for all }t\ge s
\end{equation}
where $\Psi(t,s)$ is the principal matrix solution of \eqref{eq:linearNAODE} at $s\in\R$. In this case, one also says that \eqref{eq:linearNAODE} is exponentially stable. On the other hand, \eqref{eq:linearNAODE} is asymptotically uniformly stable in backward time if~there are constants $K>0$ and $\gamma>0$ such that,
\begin{equation}\label{eq:exp-compl-unstable-linear}
\|\Psi(t,s)\|\le Ke^{\gamma(t-s)},\quad\text{for all }s\ge t
\end{equation}

Given a nonlinear problem $\dot x=f(t,x)$ (possibly of Carathéodory type), where $f$ is continuously differentiable in $x$ for almost every $t\in\R$, a globally defined solution $\widetilde x$ is said to be {\em hyperbolic attracting\/} (resp.~\emph{hyperbolic completely repelling}) if the corresponding variational
equation $\dot z= J_f\big(t,\widetilde x(t)\big)z$ satisfies an inequality of the type \eqref{eq:exp-stable-linear} (resp.~\eqref{eq:exp-compl-unstable-linear}). Note that hyperbolic attracting solutions are both (locally) pullback and forward attracting. Hyperbolic saddle solutions are also possible and this requires a generalization of the previous formulas through the concept of exponential dichotomy. This is not required for this work and therefore we point the interested reader for example to Ref.~\onlinecite{kloeden2011nonautonomous,potzsche2011nonautonomous} for the details and further references.
An important feature of hyperbolic solutions is that they persist under small perturbations\cite{potzsche2011nonautonomous}.

\subsection{On Floquet's theory and normally hyperbolic integral manifolds}\label{sec:floquet_hyperbolic}
In this section, we shall briefly relate  the Floquet exponents of a periodic solution $\gamma$ of a planar system to the normal hyperbolicity of the integral manifold associated to $\gamma$ when the planar system is embedded in the extended phase space  $\R\times\R^2$.

Let us start by recalling Floquet's Theorem in the planar case. Every fundamental matrix solution  $Y:\R\to\R^{2\times 2}$ of the $T$-periodic linear problem $\dot y=L(t) y$, has the form $Y(t)=P(t)e^{tB}$, where
$P(t)$ and $B$ are $2\times 2$ matrices, $P(t+T)=P(t)$ for all $t\in\R$ and $B$ is constant. 

A monodromy matrix of  $\dot y=L(t) y$ is a nonsingular matrix $M$ associated with a
fundamental matrix solution $Y(t)$ of $\dot y=L(t) y$ through the relation $Y(t+T)=Y(t)M$. The eigenvalues $\mu$ of a monodromy matrix are called the characteristic (or Floquet) multipliers of $\dot y=L(t) y$ and any $\lambda$ such that $\mu=e^{\lambda T}$ is called a characteristic (or Floquet) exponent of $\dot y=L(t) y$. Note that the characteristic exponents are not uniquely
defined, albeit their real parts are, and they do not depend on the chosen monodromy matrix~\cite{hale1969ordinary}. 

We shall now turn our attention to non-linear autonomous planar problems subjected a to a small time-dependent perturbation. Consider the system of equations 
\begin{equation}\label{eq:AUT-PLANAR}
\begin{split}
  \dot x&=f(x,y)+\delta f_0(t,x,y)\\
  \dot y&=g(x,y)+\delta g_0(t,x,y)   
\end{split} 
\end{equation}
where $f,f_0,g,g_0:\R^2\to\R$ are smooth. At first, consider $\delta=0$ and assume that a non-constant T-periodic solution $\gamma:\R\to\R^2$ exists. The Floquet exponents of the variational equation along $\gamma$ are\cite{hale1969ordinary} 
\[
\lambda_1=1\quad \text{and}\quad \lambda_2=\exp\int_0^T\left(\frac{\partial f}{\partial x}\big(\gamma(t)\big)+\frac{\partial g}{\partial y}\big(\gamma(t)\big)\right)\,dt,
\]
and thus $\gamma$ is asymptotically stable if  $\lambda_2<0.$
In such a case, a constant $c>0$ exists such that for any $(x_0,y_0)$ with $d\big((x_0,y_0),\gamma(\R)\big)<c$ there is $\tau=\tau(x_0,y_0)$ such that the solution $\phi_0:\R\to\R^2$ with initial condition  $(x_0,y_0)$ satisfies $|\phi_0(t)-\gamma(t-\tau)|\le ce^{\lambda_2 t}$. 

We shall now look at \eqref{eq:AUT-PLANAR} for $\delta=0$ in the augmented phase space. In view of the above, the cylinder in $\R\times\R^2$,
\[
S_0=\{(t,x,y)\mid (x,y)=\gamma(\theta),\, \theta\in[0,T],\,t\in\R\} 
\]
and the torus in $\mathbb{S}^1\times\R^2$,
\[
T_0=\{(t,x,y)\mid (x,y)=\gamma(\theta),\, \theta,t \text{ mod } T\in[0,T)\},
\]
with $\tau>0$, are invariant  for the flow induced by \eqref{eq:AUT-PLANAR} with $\delta=0$, and asymptotically stable. In their respective spaces, they are examples of integral manifolds\cite{hale1969ordinary}, i.e., invariant sets for the flow induced by \eqref{eq:AUT-PLANAR} with a special structure. 
Considering a change of coordinates $(x,y)=\gamma(\theta)+Z(\theta)\rho$ (a moving orthonormal system along $\gamma$ as described in Theorem VI.1.1 in \onlinecite{hale1969ordinary}), the solutions of the original planar system can be described in a neighborhood of $S_0$ (or~$T_0$) by
\[
\begin{split}
\dot \theta&=1+\Theta(\theta,\rho)\\  
\dot \rho&=\lambda_2\rho+R(\theta,\rho),
\end{split}
\]
where $\theta$ is a radial coordinate, whereas $\rho$ represents the normal deviation from $\gamma(\theta)$, $\Theta(\theta,\rho)=O(|\rho|)$ and $R(\theta,\rho)=O(|\rho|^2)$ as $\rho\to0$, and $\Theta$ and $R$ are $T$-periodic in $\theta$. Details on this construction can be found in Section VII.1 of Ref.~\onlinecite{hale1969ordinary}.

Importantly, such an integral manifold is robust against perturbations: let $\delta\neq0$ in \eqref{eq:AUT-PLANAR},
where $f_0$ and $g_0$ are bounded in a neighborhood of $S_0$ uniformly in $t$, and $\delta$ is sufficiently small. Then, the above change of coordinates applied to the perturbed problem leads to
\[
\begin{split}
\dot \theta&=1+\Theta(\theta,\rho)+\delta\Theta_0(t,\theta,\rho)\\  
\dot \rho&=\lambda_2\rho+R(\theta,\rho)+\delta R_0(t,\theta,\rho),
\end{split}
\]
where $\Theta_0(t,\theta,0)$ and $R_0(t,\theta,0)$ are bounded.
Thus, Corollary VII.2.1 in~\onlinecite{hale1969ordinary} ensures the existence of a continuous function $r:\R^2\times [0,\delta_0]\to\R^+ $, $(t,\theta,\delta)\mapsto r(t,\theta,\delta)$, for some $\delta_0>0$, bounded and Lipschitz continuous in $\theta$, so that 
\[
S_\delta=\{(t,\theta,r(t,\theta,\delta))\mid  \theta\in[0,T],\,t\in\R\} 
\]
is an integral manifold of \eqref{eq:AUT-PLANAR} in $\R\times\R^2$.
Moreover, if the functions $f_0,g_0$ are $\tau$-periodic or almost periodic in $t$, the same holds for  $r(\cdot,\theta,\delta)$.

\section{Description of the global attractor}
\label{sec:global-attractor}

\subsection{Existence of a global attractor for all $\delta\in[0,1]$}\label{subsec:existence-glob-attr}
Our first result proves the dissipative nature of~\eqref{eq:problem} which gives rise to a pullback and a global attractor for the system.

\begin{thm}\label{prop:globAttr} Let $v:\R\to\R$ be essentially bounded. Then, the following statements are true for~\eqref{eq:problem} for any $\delta\in[0,1]$:
\begin{itemize}
\item[\rm(i)]  there exists a unique bounded pullback attractor $\widehat A^\delta=\{ A^\delta_w\subset\R^2\!\mid w\in \mathcal{H} (v)\}$ for the skew-product flow and it is defined by
    \[A^\delta_w=\bigcap_{\tau\ge 0}\; \overline{\bigcup_{t\ge \tau}x(t,w_{-t}, B_{\overline{r}};\delta)}\quad \text{for each }\; w\in\mathcal{H} (v), \]
    where $B_{\overline r}$ is the closed ball in $\R^2$ of radius $\overline r$ centered at the origin;
\item[\rm(ii)] there is a global attractor of the skew-product flow 
    \[ \mathcal A^\delta=\bigcap_{\tau\ge 0}\; \overline{\bigcup_{t\ge \tau} \Phi_\delta(t, \mathcal{H} (v)\times B_{\overline{r}})}\,=\!\bigcup_{w\in \mathcal{H} (v)}\!\left\{\{ w\}\times A^\delta_w\right\}\,.\]
\end{itemize}
\end{thm}
\begin{proof}
Firstly, note that, thanks to Proposition~\ref{prop:SKEW-PROD-COMPACT}, \eqref{eq:problem} induces a continuous skew-product flow $\Phi_\delta:\mathcal{U}\subset\R\times\mathcal{H}(v)\times \R^2\to\mathcal{H}(v)\times \R^2$ on the compact metric space $\mathcal{H}(v)$. Moreover, since $v$ is essentially bounded, then every element of $\mathcal{H}(v)$ shares the same norm in $L^\infty$. 

For any $w\in\mathcal{H}(v)$, considering the induced dynamical systems in polar coordinates, we have that the radial equation satisfies
\[
\begin{split}
r'= r^3\Bigg[&-\frac{\cos^4(\theta)}{3} +\frac{\delta w(t)+a \ep \sin(\theta)}{r^3} \\
& + \frac{\cos^2(\theta)-\ep\sin^2(\theta)+(1-\delta-\ep)\sin(\theta)\cos(\theta)}{r^2} \,\Bigg].
\end{split}
\]
Hence, there is $\overline r=\overline r(a, \ep, v)$ such that $r'<0$ whenever $r\ge\overline r$  uniformly on $w\in\mathcal{H}(v)$, $\delta\in[0,1]$ and $t\in\R$.
Consequently, the flow $\Phi$ is uniformly bounded dissipative in the sense that all the solutions of~\eqref{eq:problem} are defined on positive half-lines and uniformly ultimately bounded by the sphere of radius $r=\overline r$. 
Therefore, among other references, (i) follows from Theorem~3.20 in ~Ref.~\onlinecite{kloeden2011nonautonomous}. 
In turn, the existence of a global attractor $\mathcal A^\delta$ follows from the compactness of $\mathcal{H}(v)$ and from~Theorem 2.2 in Ref.~\onlinecite{cheban2002relationship} and, as shown in Theorem 16.2 in~Ref.~\onlinecite{carvalho2012attractors}, $A^\delta_w$ is the section of $\mathcal A^\delta$ over $w$, that is,
\begin{equation*}
{\mathcal A^\delta}=\bigcup_{w\in\mathcal{H}(v)}\left\{\{ w\}\times A^\delta_w\right\}\, ,
\end{equation*}
which finishes the proof.
\end{proof}

In the next sections, we aim to give a finer description of the global attractor. We shall firstly focus on the boundary cases $\delta=0$ and $\delta=1$ and nearby values of $\delta$.

\subsection{Global attractor for $\delta=0$ (FitzHugh-Nagumo model) and nearby values}
\label{subsec:global-attractor-d=0}


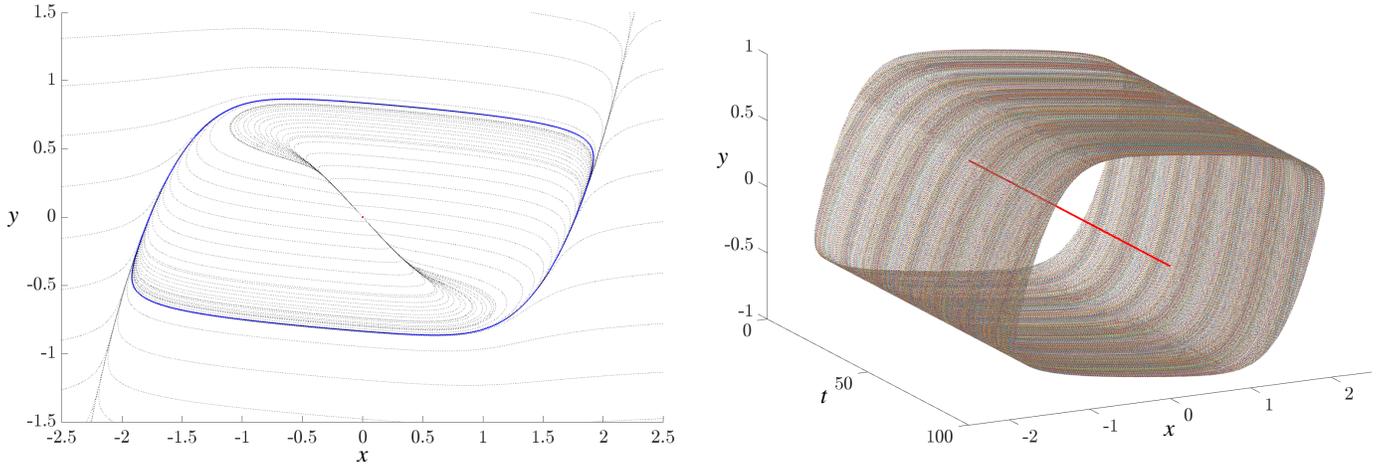
\begin{figure*}
    \centering
\begin{overpic}[width=0.48\textwidth]{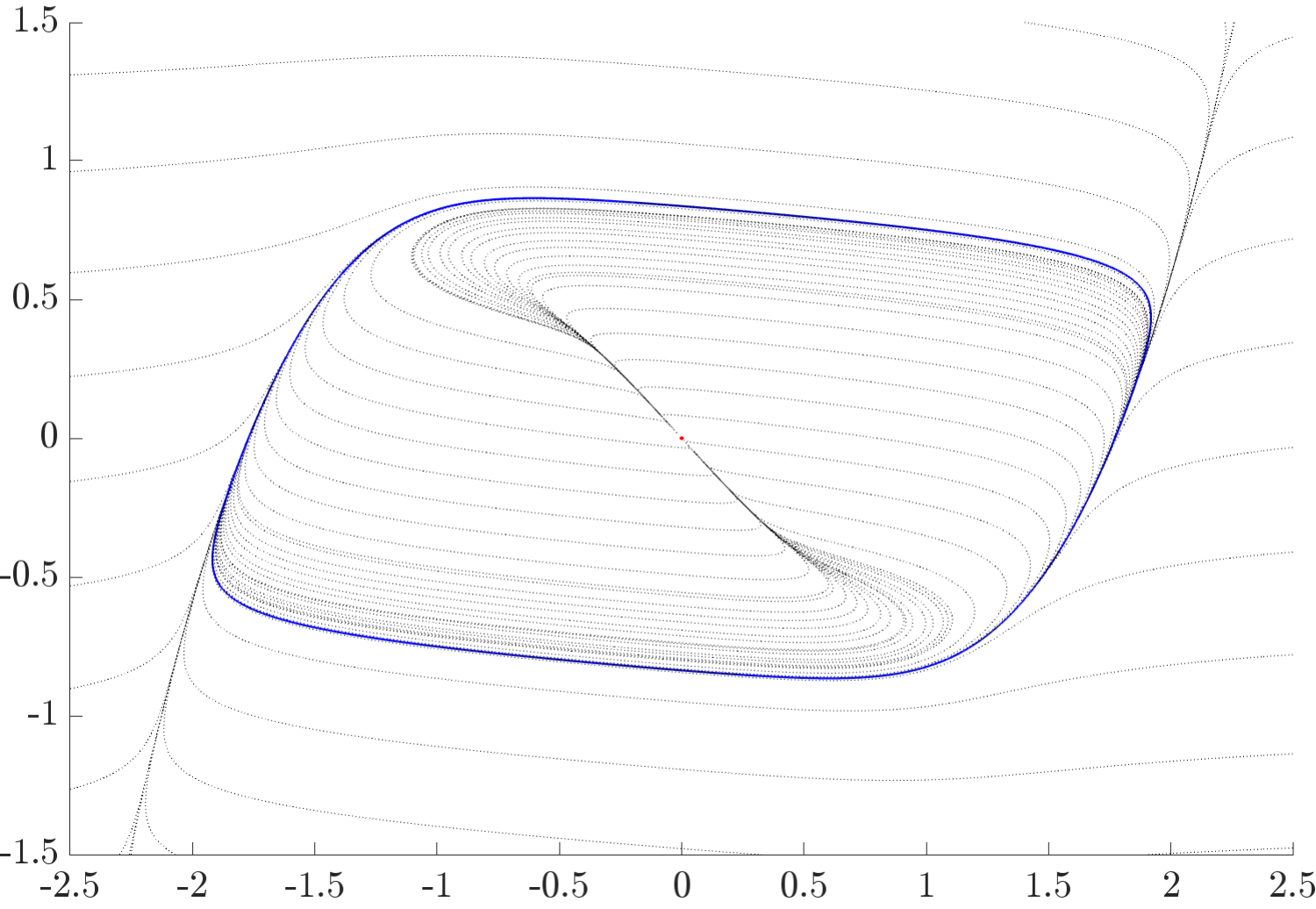} 
\put(51,-2){$x$}
\put(-3,35){$y$}
\end{overpic}
\hspace{5mm}
\begin{overpic}[width=0.48\textwidth]{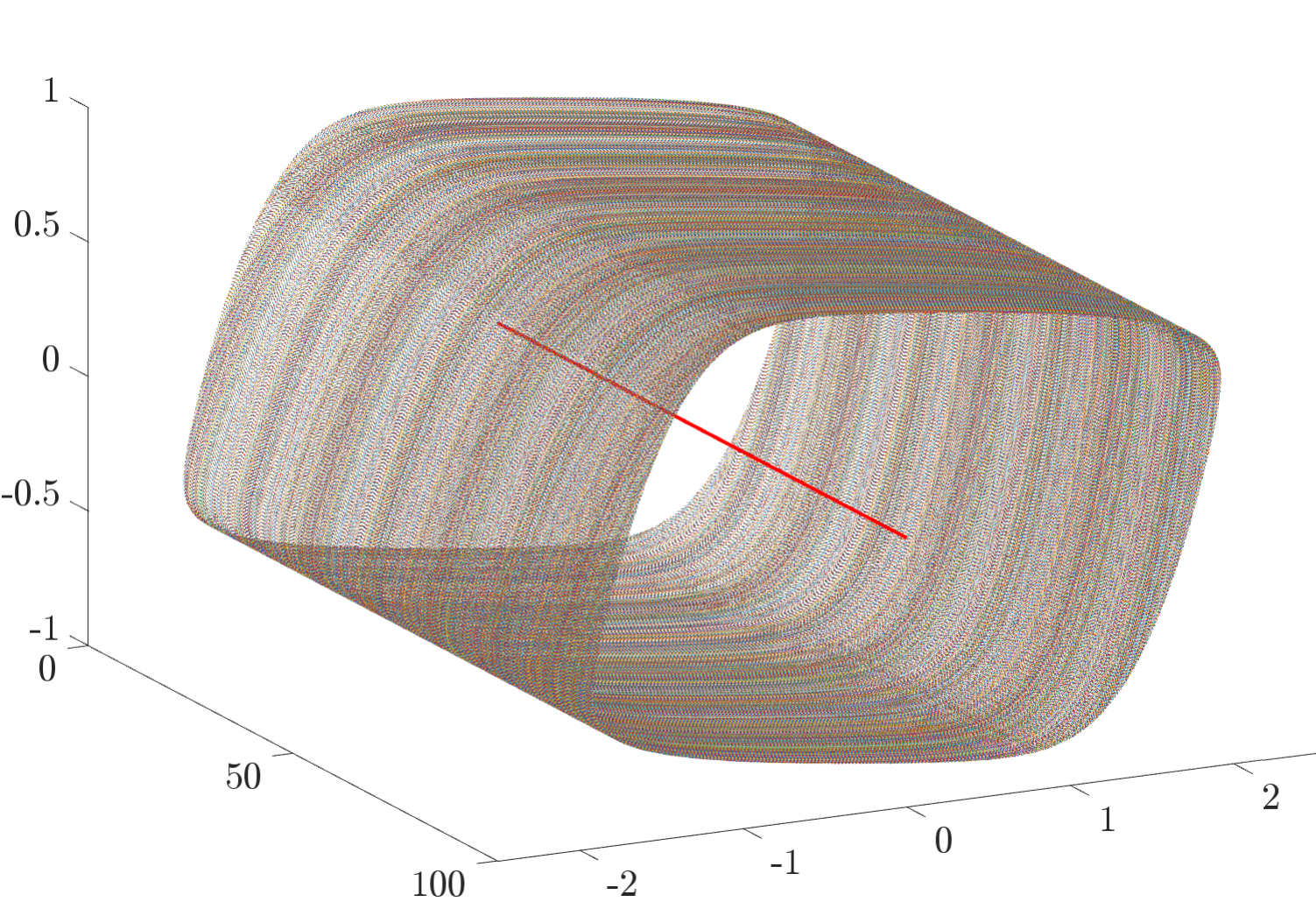} 
\put(15,7){$t$}
\put(68,2){$x$}
\put(-1,44){$y$}
\end{overpic}
\caption{On the left-hand side, dynamics of the autonomous FitzHugh-Nagumo model~\eqref{eq:FHN} for $a=0$, $b=0.7$ and $\ep=0.1$. The attracting periodic orbit is depicted in blue. The unstable equilibrium in red. The black dotted lines represent sample trajectories. On the right-hand side, the system is portrayed in the extended phase space $\R\times\R^2$ giving rise to a trivial integral manifold.}
\label{fig:per-orbit-FHN}
\end{figure*}

We briefly recall the construction of an attracting limit cycle for the FitzHugh-Nagumo model via geometric singular perturbation theory. Equation~\eqref{eq:FHN} is generally referred to as in its fast time-scale as opposed to its slow time-scale obtained via the change of variable $\tau=\ep t$,
\begin{equation}\label{eq:vdp_slow}
\begin{sistema}
\ep\dot  x = y-\frac{x^3}{3} +x\\
\dot y =a-x-by.\\
\end{sistema}
\end{equation}
By setting $\ep=0$ in~\eqref{eq:vdp_slow}, we obtain a reduced (or slow) system with an algebraic constraint 
\begin{equation}\label{eq:vdp_slow_0}
\begin{sistema}
0= y-\frac{x^3}{3} +x\\
\dot y =a-x-by.\\
\end{sistema}
\end{equation}
The algebraic constraint 
\[
S=\{(x,y)\in\R^2\mid y=x^3/3-x\},
\]
is called the \emph{critical manifold} of~\eqref{eq:vdp_slow}.
Note also that the points of $S$ are equilibrium points for the so-called layer problem (or fast subsystem), which is obtained by setting  $\ep=0$ in the fast time-scale problem~\eqref{eq:FHN},
\begin{equation}\label{eq:vdp_fast_0}
\begin{sistema}
x' = y-\frac{x^3}{3} +x\\
y' =0.\\
\end{sistema}
\end{equation}
Fenichel theory~\cite{fenichel1979geometric} guarantees the persistence of any compact connected and normally hyperbolic subset of $S$ when $\ep>0$ is small enough. 
Recall that an invariant manifold is called normally hyperbolic if, loosely speaking, the attraction to the manifold in forward and/or backward time is stronger than the dynamics on the manifold itself; see~Ref.~\onlinecite{fenichel1971persistence,fenichel1974asymptotic,fenichel1977asymptotic,hirsch2006invariant,kuehn2015multiple} for the exact definition. In the case of fast-slow system like ours, normal hyperbolicity of a manifold can be verified in a simpler way: if $M$ is a subset of $S$, it is sufficient to show that all the points in $M$ are hyperbolic equilibria of the layer problem.

Therefore, since $\frac{\partial}{\partial x}f(x,y)=1-x^2\neq0$ for all $x\in\R\setminus\{\pm1\}$, the critical manifold is normally hyperbolic except for $p_\pm=(\pm1,\mp 2/3)$. Specifically, the sign of $\frac{\partial}{\partial x}f(x,y)$ characterizes the branches $S\cap\{(x,y)\in\R^2\mid |x|>1\}$ as attracting and the branch $S\cap\{(x,y)\in\R^2\mid |x|<1\}$ as repelling, regardless of the value of $a\in\R$ and $b\ge0$. The so-called slow flow, that is, the dynamics of the reduced system~\eqref{eq:vdp_slow_0} can be obtained via the invariance equation
\[
\begin{split}
a-x-by(x)&=\frac{d}{d\tau}y=\frac{d}{d\tau}\left(\frac{x^3}{3} -x\right)=(x^2-1)\frac{d}{d\tau}x\\
&\Rightarrow\ \ \dot x =\frac{a-x-b(x^3/3-x)}{x^2-1}.
\end{split}
\]
The slow flow is well-defined everywhere except for $x=\pm 1$ (assuming that $a\neq\pm(1-2b/3)$). 
For fixed parameters satisfying $|a|<1-2b/3$, the slow flow is such that $\dot x>0$ for $x<1$ and $\dot x<0$ for $x>1$.
A desingularization at $x=\pm 1$ can be obtained by rescaling time with the factor $(x^2-1)$ although at a cost, the reversal of time direction on the repelling branch of the critical manifold. We will momentarily ignore this issue as it has no effect on what comes next. 
Consider the \emph{candidate closed orbit} obtained by concatenating orbits on the attracting branches of the critical manifolds with quick horizontal jumps from one of the attracting branches of $S$ onto the other attracting branch of $S$ at the fold points $p_\pm=(\pm1,\mp 2/3)$. 
In other words, the candidate orbit follows the slow flow on $S$  up to a fold point and then it quickly transitions via a trajectory of the layered problem~\eqref{eq:vdp_fast_0} to the opposite attracting branch of $S$ and so on. 
Under some genericity conditions, which for \eqref{eq:vdp_slow} read as,
\begin{equation}\label{eq:H1-FHN}
a-1+2b/3\neq0\quad\text{and}\quad a+1-2b/3\neq0,
\end{equation}
and some assumptions on the slow flow, which for \eqref{eq:vdp_slow} read as,
\begin{equation}\label{eq:H2-FHN}
a-x-b(x^3/3-x)
\begin{cases}
    >0&\text{if } x<-1,\\
    <0 &\text{if } x>1,
\end{cases}
\end{equation}
\begin{figure}
    \centering
\begin{overpic}[width=0.48\textwidth]{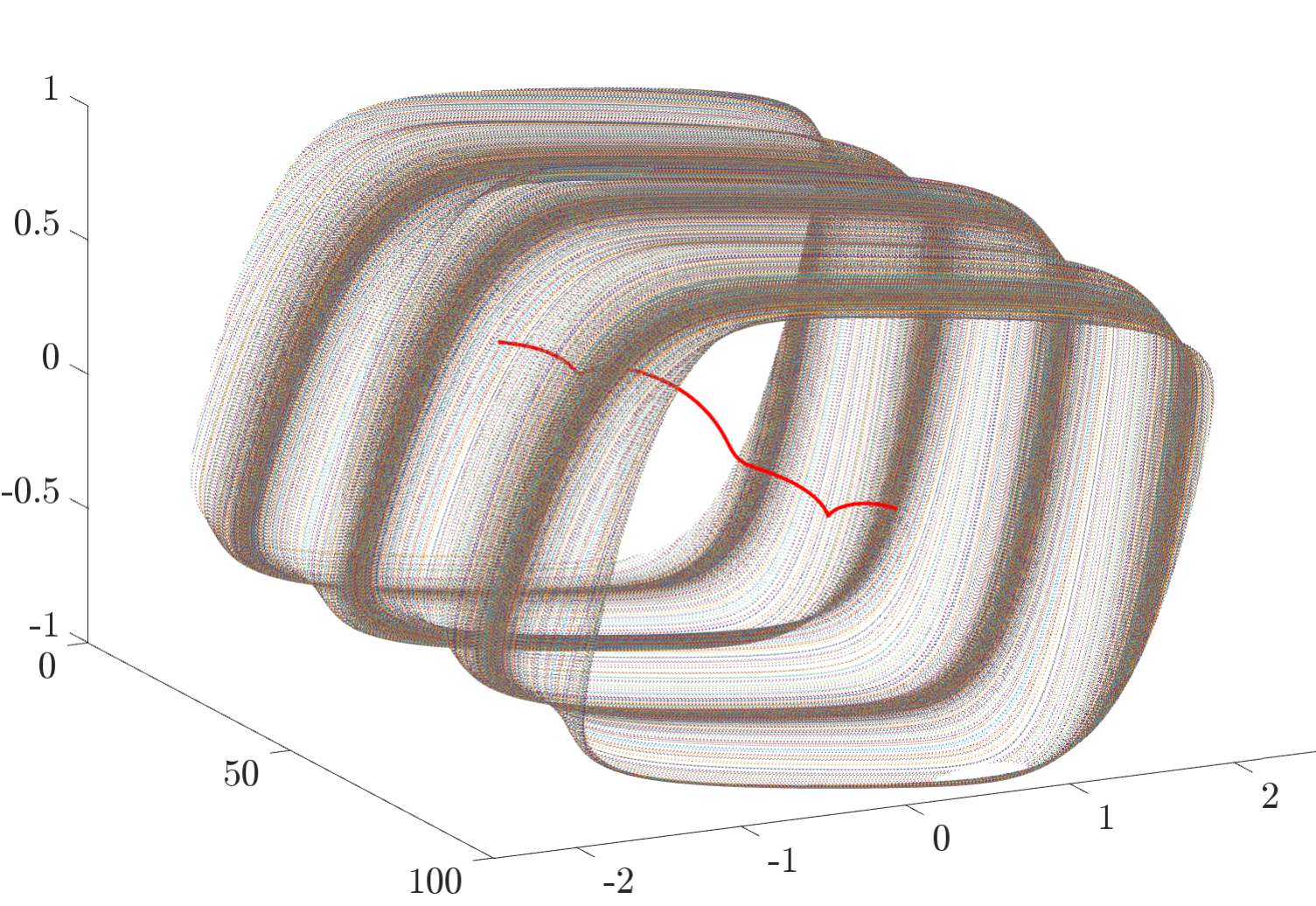} 
\put(15,7){$t$}
\put(68,2){$x$}
\put(-1,44){$y$}
\end{overpic}
\caption{Numerical approximation of the integral manifold of~\eqref{eq:problem} for $\delta=0.07$, $a=0$, $b=0.7$, $\ep=0.1$ and quasi-periodic forcing $v(t)=\cos(2\pi t/30)+\sin(2\pi t/(30\sqrt{5}))$. The numerically approximated hyperbolic completely unstable trajectory is depicted in red.}
\label{fig:perturbed-FHN}
\end{figure}
Theorem 2.1 in Ref.~\onlinecite{krupa2001relaxation} shows that, for $\ep>0$ sufficiently small, the candidate orbit perturbs into a strongly attracting periodic orbit of \eqref{eq:vdp_slow} that lies $O(\ep^{2/3})$-close to the candidate and approaches it in Hausdorff distance as $\ep\to0$. 
Since in our analysis $\ep>0$ is kept fixed, we shall denote this orbit by $\gamma_0(t)=\big(x_0(t),y_0(t)\big)$, where the subindex $0$ refers to the fact that we are dealing with the case $\delta=0$. 
The term strongly attracting refers to the fact that its non-unitary Floquet exponent is negative. 

As we have seen in Section~\ref{sec:floquet_hyperbolic}, this means that, whenever $\ep>0$ is small enough, the attracting periodic orbit of \eqref{eq:FHN} induces an attracting integral manifold $S_0$ for the process associated to it on the augmented phase space $\R\times\R^2$. In particular, by construction, it can be shown that $S_0$ is homeomorphic to $\R\times\mathbb{S}^1$ since the graph of the periodic limit cycle $\gamma_0(t)$ can be parametrized by the angle $\theta\in[0,2\pi]$ modulo $2\pi$. Figure~\ref{fig:per-orbit-FHN} depicts the strongly attracting periodic orbit of \eqref{eq:vdp_slow} for $a=0$, $b=0.7$ and $\ep=0.1$ and the induced integral manifold in $\R\times \R^2$.  

Furthermore, this integral manifold persists under small perturbations\cite{hale1969ordinary}. Hence, for $\delta>0$  sufficiently small, an attractive integral manifold $S_\delta$---also homeomorphic to $\R\times\mathbb{S}^1$ exists for \eqref{eq:problem}.
In accordance to Sec.~\ref{sec:floquet_hyperbolic}, the equation on $S_\delta$ reduces to 
\[
\dot \theta=1+\delta\Theta_0(t,\theta,\delta).
\]
In the extended notation of the skew-product flow of Section \ref{subsec:skew-prod}, we can consider the integral manifold $S_\delta$ parametrized with respect to the elements of $\Hu(v)$, that is,
\[
S_\delta=\{(w,\theta,r_\delta(w,\theta))\mid \theta\in[0,2\pi],\,w\in\Hu(v)\} 
\]
This approach allows to interpret the flow $\varphi$ on $\Hu(v)\times S_\delta$---that is the restriction of the skew-product flow $\Phi$ to $\Hu(v)\times S_\delta$---as a circle extension of the base flow $\sigma_t$ on $\Hu(v)$ in the sense that, 
\[
\pi\circ\varphi(t,w,p)=\sigma_t\circ\pi(w,p),
\]
where $\pi:S_\delta\to\Hu(v)$ is the natural projection $(w,p)\mapsto \pi(w,p)=w$. In fact, circle flows arising from almost periodically forced, damped nonlinear oscillators are a well-studied topic \cite{huang2009almost, yi2004almost,romeiras1987strange,pikovsky2006strange, jager2006denjoy,jager2006towards}. 
The resulting dynamical scenarios can be fairly complicated and certainly exceeding the complexity of the base flow due to the interaction between internal and forcing frequencies. 
For example, if the model is periodically forced, it follows from the classical Poincaré–Birkhoff–Denjoy theory that a minimal set for $\varphi$ can be periodic, almost periodic or almost automorphic with a Cantor structure. Under almost periodic forcing, the dynamics on $\Hu(v)\times S_\delta$ can be even more complicated. Ref.~\onlinecite{huang2009almost} provides extended information in this regard.  In particular, even in quasi-periodically forced nonlinear oscillators as simple as Van der Pol, numerical studies have shown the existence of so-called strange non-chaotic attractors, that is invariant sets which have a fractal-like geometric structure but admit no positive Lyapunov exponent\cite{grebogi1984strange,romeiras1987strange}.

In order to complete our analysis of the global attractor for \eqref{eq:problem} when $\delta$ is close to zero, we note that under the assumptions
\eqref{eq:H1-FHN} and \eqref{eq:H2-FHN}, the autonomous FitzHugh-Nagumo model \eqref{eq:FHN} has also one hyperbolic completely unstable equilibrium point which lies inside the periodic orbit in $\R^2$. Note that an unstable hyperbolic equilibrium for an autonomous differential equation is in particular a completely unstable hyperbolic solution in the nonautonomous sense presented at the end of Section \ref{subsec:attractors}.
In turn, the (nonautonomous notion of) hyperbolicity guarantees robustness against small perturbations\cite{potzsche2011nonautonomous}. Hence, for any $f_0,g_0$ bounded and $\delta\neq0$ sufficiently small, the hyperbolic equilibrium is perturbed into a hyperbolic completely unstable trajectory for the flow induced by \eqref{eq:AUT-PLANAR}. Figure \ref{fig:perturbed-FHN} shows the persisting integral manifold and hyperbolic repelling solution for $\delta=0.07$ and same values for the other parameters as in Figure \ref{fig:per-orbit-FHN}.

\subsection{Global attractor  for $\delta=1$ (non-autonomous skewed problem) and nearby values}\label{subsec:global-attract-d=1}

In this subsection, we show the existence of a hyperbolic attracting solution for the  problem~\eqref{eq:NA_limit}. 
First of all, note that the dynamical scenarios of the first equation in~\eqref{eq:NA_limit} have been exhaustively studied by Due\~nas et al.~\cite{duenas2022generalized}. 
In particular, for any $v\in C(\R,\R)$ bounded and uniformly continuous, there is at least one, and at most two, hyperbolic attracting solutions of $\dot x = v(t)-\frac{x^3}{3} +x$, which are copies of the base, i.e., if $v$ is $\tau$-periodic (or almost periodic), the same holds for the hyperbolic attracting solutions. 
We shall be working under the assumption that there is exactly one hyperbolic attracting solution $\phi_1:\R\to\R$ for $\dot x = v(t)-\frac{x^3}{3} +x$, that we shall denote by $\phi_1:\R\to\R$. In this case, in fact, $\phi_1$ is also globally attracting. 
As explained in Subsection \ref{subsec:attractors}, hyperbolicity of $\phi_1$ means that there are $K\ge1$ and $\alpha>0$ such that
\begin{equation}\label{eq:EXP-DICH-hyp-sol-cubic}
  \psi(t,t_0):=\exp\Big(\int_{t_0}^t\big(1-(\phi_1(s))^2\big)\, ds\Big)\le Ke^{-\alpha(t-t_0)}.  
\end{equation}
By plugging $\phi_1(t)$ into~\eqref{eq:problem}, and integrating the scalar linear differential equation $y'=\ep(a-\phi_1(t)-by)$, with $b>0$, for initial datum $y_0$ at time $t_0$, we obtain the following solution to \eqref{eq:NA_limit} defined for all $t>t_0$, 
\[
\begin{sistema}
x(t)=\phi_1(t)\\
\begin{split}
\displaystyle y(t,t_0,y_0) =\frac{a}{b}+e^{-b\ep(t-t_0)}\left(y_0-\frac{a}{b}\right)&-\ep\int_{t_0}^te^{-b\ep(t-s)}\phi_1(s)\,ds.\\    
\end{split}
\end{sistema}
\]

The pullback limit as $t_0\to-\infty$ gives  the bounded solution $\phi:\R\to\R^2$ defined by,
\[
\begin{sistema}
x(t)=\phi_1(t)\\
\displaystyle y(t) =\frac{a}{b}-\ep\int_{-\infty}^te^{-b\ep(t-s)}\phi_1(s)\,ds=:\phi_2(t).\\
\end{sistema}
\]
We shall show that, this is a hyperbolic solution of \eqref{eq:NA_limit}. The variational equation of \eqref{eq:NA_limit} along $\phi$
\[
\dot z= \begin{pmatrix}
    1-(\phi_1(t))^2&0\\
    -\ep&-b\ep
\end{pmatrix}z
\]
admits the principal matrix solution at $t_0\in\R$,
\[
\Psi(t,t_0)=\begin{pmatrix}
    \psi(t,t_0)&0\\
    -\ep\int_{t_0}^t e^{-b\ep(t-s)}\psi(s,t_0) \,ds&
    e^{-b\ep(t-t_0)}
\end{pmatrix}
\]
Now note that, using \eqref{eq:EXP-DICH-hyp-sol-cubic}, 
\[
\begin{split}
&e^{-b\ep(t-t_0)}+\ep\int_{t_0}^t e^{-b\ep(t-s)}\psi(s,t_0) \,ds\\
&\qquad\le e^{-b\ep(t-t_0)}+\frac{\ep K}{b\ep-\alpha} \big(e^{-\alpha (t-t_0)}- e^{-b\ep(t-t_0)\big)}
\end{split}
\]
and for any $b>0$ and $\alpha>0$, we can choose $\ep>0$ sufficiently small so that $b\ep-\alpha<0$ and the right-hand side of the previous inequality is smaller than $e^{-b\ep(t-t_0)}$. Therefore,
\[
\|\Psi(t,t_0)\|_\infty\le \ Ke^{-b\ep(t-t_0)},
\]
which proves the sought-for hyperbolicity of $\phi(t)=(\phi_1(t),\phi_2(t))$. Moreover, it is immediate to prove that this entire bounded solution is globally attractive. Being hyperbolic, the solution $\phi$ perturbs into a hyperbolic attracting solution of \eqref{eq:problem} for $\delta$ sufficiently close to $1$\cite{potzsche2011nonautonomous}.

Figure \ref{fig:hyp-sol} shows the numerical approximation of the unique hyperbolic attracting solution to \eqref{eq:NA_limit} (upper panel) and its persistence upon decreasing $\delta$ (lower panel). The rest of parameters are as in Figures \ref{fig:per-orbit-FHN} and \ref{fig:perturbed-FHN}.

Note also that, if  $v$ is $\tau$-periodic, then using the change of variable $s=u+\tau$ in 
\[
\begin{split}
\phi_2(t+\tau)&=\frac{a}{\ep}+\int_{-\infty}^{t+\tau}e^{-\ep(t+\tau-s)}\phi(s)\,ds\\
&=\frac{a}{\ep}+\int_{-\infty}^te^{-\ep(t-u)}\phi(u+\tau)\,du\\
&=\frac{a}{\ep}+\int_{-\infty}^te^{-\ep(t-u)}\phi(u)\,du=\phi_2(t)(t)
\end{split}
\]
yields the required periodicity for $\phi$.

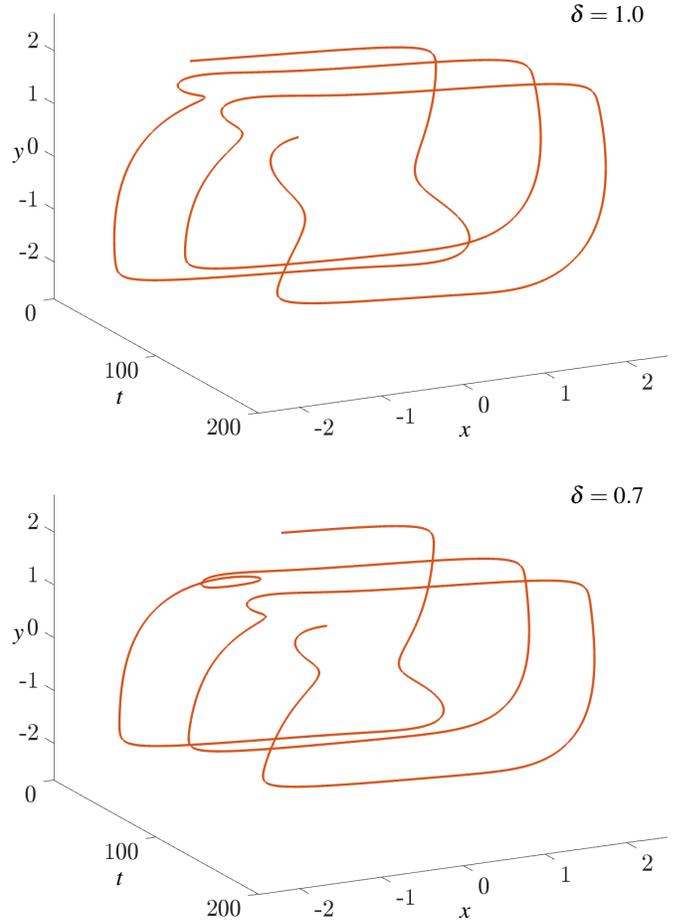
\begin{figure}
    \centering
\begin{overpic}[width=0.48\textwidth]{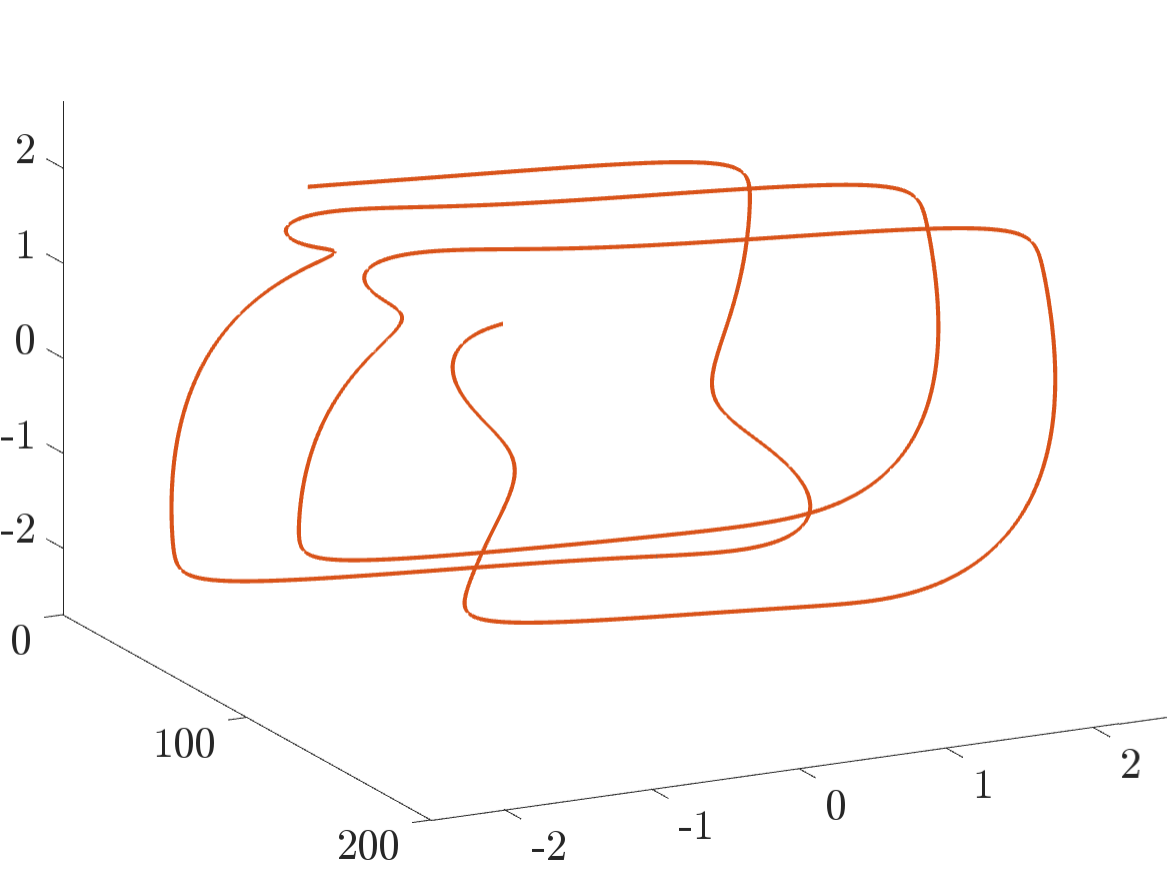} 
\put(85,65){$\delta=1.0$}
\put(15,6){$t$}
\put(68,1){$x$}
\put(-1,44){$y$}
\end{overpic}
\begin{overpic}[width=0.48\textwidth]{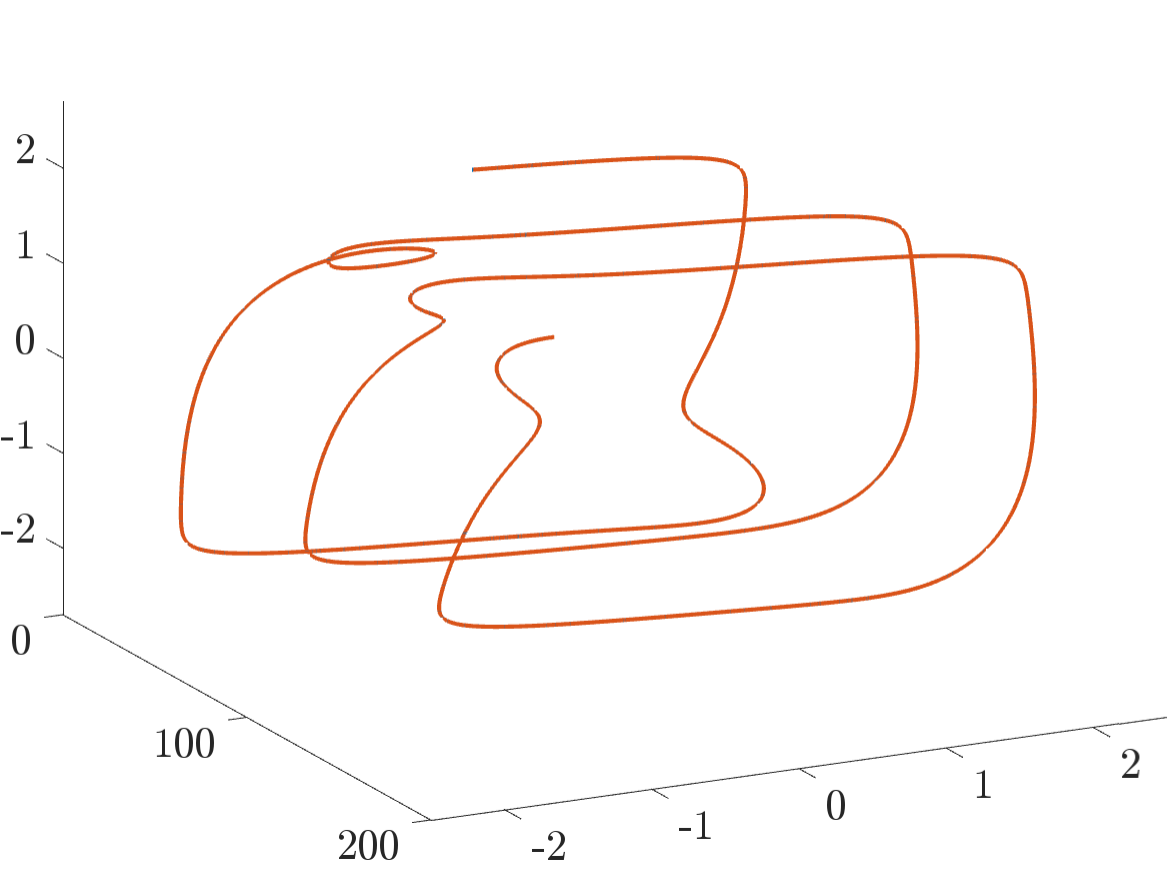} 
\put(85,65){$\delta=0.7$}
\put(15,6){$t$}
\put(68,1){$x$}
\put(-1,44){$y$}
\end{overpic}
\caption{Upper panel: numerical approximation of the hyperbolic attracting solution $\phi$ for \eqref{eq:problem} with $\delta=1$, $a=0$, $b=0.7$, $\ep=0.1$ and almost periodic forcing $v(t)=\cos(2\pi t/30)+\sin(2\pi t/(30\sqrt{5}))$ as in Figure \ref{fig:perturbed-FHN}. Lower panel: the same problem with $\delta=0.7$. The orbits were obtained by numerically integrating \eqref{eq:problem} for three hundred initial conditions starting at uniformly spaced initial times in $[-200,-100]$ and then discarding the transient till $t=0$.}
\label{fig:hyp-sol}
\end{figure}

\subsection{Global attractor for intermediate values of $\delta$}\label{subsec:global-attractor-int-d}
The case of intermediate values of $\delta$ away from 0 and 1 is considerably more complicated and highly dependent on the choice of $v$. Nonetheless, we are able to investigate the finite-time behaviour of singularly perturbed problems featuring an explicit dependence on the fast time such as \eqref{eq:problem} thanks to the results developed in Ref.~\onlinecite{longo2024tracking}. 
Assume $v$ is almost periodic. Then, the flow on the hull of $v$ is minimal and uniquely ergodic.
Let us also set some further notation. Write \eqref{eq:problem} in the slow time-scale
\[
\begin{sistema}
\ep x' = (1-\delta)y+\delta v(\tau/\ep)-\frac{x^3}{3} +x\\
y' =a-x-b y
\end{sistema}
\]
and call $(x_\ep(\tau),y_\ep(\tau))$ the solution of the previous problem with initial condition $(x_0,y_0)$ at time zero.
Moreover, consider the so-called layered problems obtained from \eqref{eq:problem} for $\ep=0$,
\begin{equation} \label{eq:layered-prob}
x' = (1-\delta)y+\delta v(t)-\frac{x^3}{3} +x,
\end{equation}
for $y\in [-k,k]$ with $k>0$, and $\delta\in[0,1]$. This is family of scalar nonautonomous d-concave differential equations dependent on the parameters $y\in [-k,k]$ and $\delta\in[0,1]$. Let us fix any $\delta\in[0,1]$ and remove the dependence in this parameter from the upcoming notation. For each $y\in [-k,k]$, \eqref{eq:layered-prob} induces a skew-product flow
\[
\begin{split}
   \Phi_{y}:\R\times \Hu(v)\times\R&\to\Hu(v)\times\R,\\
   (t,w,x_0)&\mapsto  \Phi_{y} (t,w,x_0)=\big(w_t, x(t,w,x_0;y)\big),
\end{split}
\]
where $x(t,w,x_0;y)$ denotes the solution at time of $t$ of \eqref{eq:layered-prob} (for a fixed pair $(y,\delta)$) with initial condition $x(0)=x_0$. Each of these flows admits a global attractor $\mathcal{A}^y=\cup_{w\in\mathcal{H}(v)}\{w\}\times A^y_w$ whose structure has been described  in detail in Ref.~\onlinecite{duenas2022generalized} and briefly recalled in Section \ref{subsec:global-attract-d=1}. The dependence on $\delta$ for $\mathcal{A}^y$ is left tacit as we will not be varying $\delta$ within this paragraph.
Numerical evidence suggests that for every $y\in[-3,3]$, the skew-product flows $ \Phi_{y}$ only have one copy of the base which corresponds to the global attractor $\mathcal{A}^y$. In other words, there is a continuous map $\eta:\Hu(v)\times[-3,3]\to\R$ such that $A^y_w=\eta(w,y)$. Consequently, the flow $\Phi_y$ has a unique invariant measure concentrated on $\mathcal{A}^y$; let's call $\mu^y$ its projection on $\R$. Then, Theorem 4.8 in Ref.~\onlinecite{longo2024tracking} allows us to draw the following conclusions. Fix $\tau_0>0$,
\begin{itemize}[leftmargin=*]
    \item for any sequence $\ep_j\to0$ there is a subsequence $\ep_k\downarrow0$ such that $y_{\ep_k}$ converges uniformly in $[0,\tau_0]$ to the unique solution $y_0(\tau)$ of the differential equation
    \[
    y'=\int_{\R} (a-x-by)\,d\mu^y(x),\quad y(0)=y_0.
    \]
    The right-hand side above can be understood as an averaged problem in the sense of Artstein\cite{artstein2007averaging} and Sanders and Verhulst\cite{sanders1985averaging}.
    \item Given $\sigma>0$, there exists $T=T(\sigma,x_0,\mathcal A)>0$ and $k_0=k_0(\sigma,T)>0$ with $2T<\tau_0/\ep_{k_0}$ such that for every $k\ge k_0$,
    \[
    |x_{\ep_k}(t)-\eta(v_t,y_0(\ep_k t))|\le \sigma,\quad\text{for all } t\in[T,\tau_0/\ep_k].
    \]
\end{itemize}

In other words, as $\ep\to0$, the slow variable converges uniformly on $[0,\tau_0]$ to the unique solution of an averaged scalar problem anchored at the union of the hyperbolic solutions of \ref{eq:layered-prob} upon the variation of $y\in[-k,k]$. The fast variable, in turn, approaches a path determined by the value at time $t\in[0,\tau_0/\ep_k]$ of the hyperbolic solution of \eqref{eq:layered-prob} for $y=y_0(\ep_k t)$.

\begin{figure*}
    \centering
\begin{overpic}[width=0.33\textwidth]{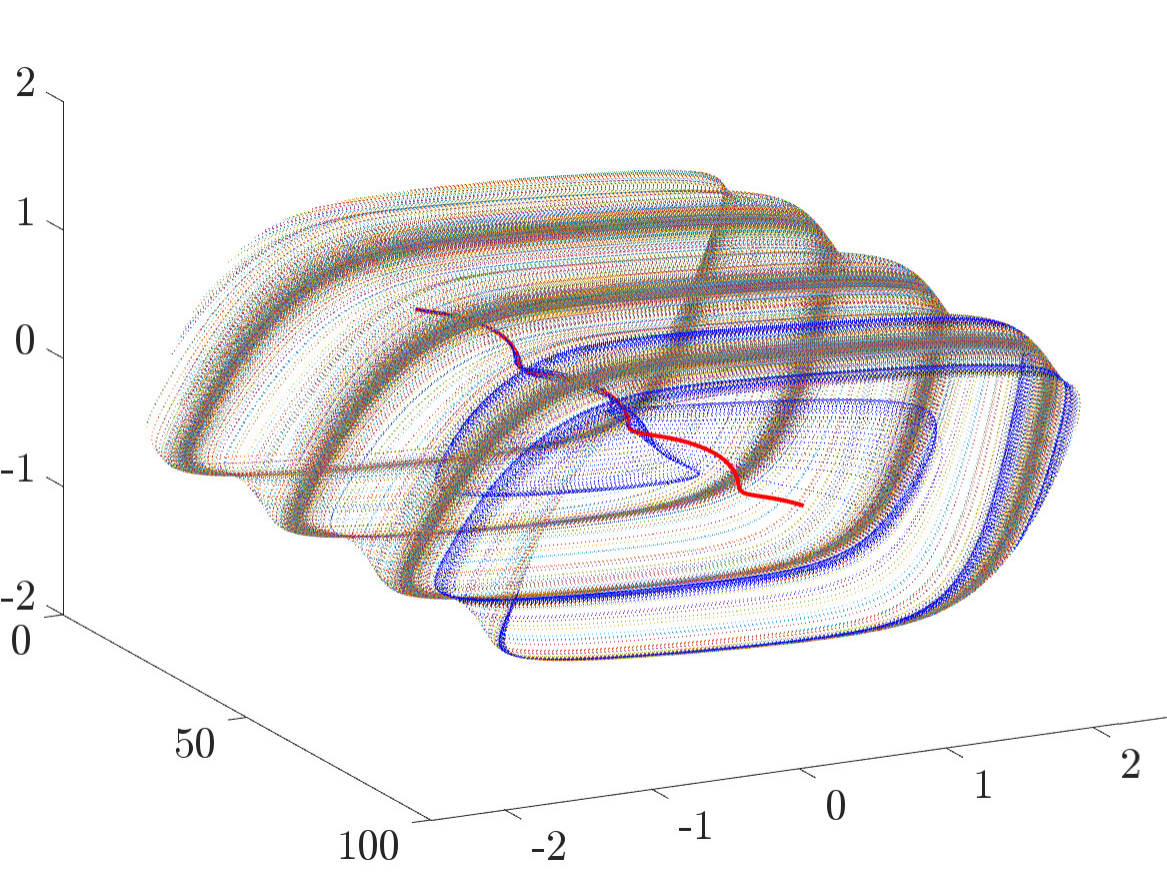} 
\put(42,72){$\delta=0.1$}
\put(15,6){$t$}
\put(70,1){$x$}
\put(-5,44){$y$}
\end{overpic}
\begin{overpic}[width=0.33\textwidth]{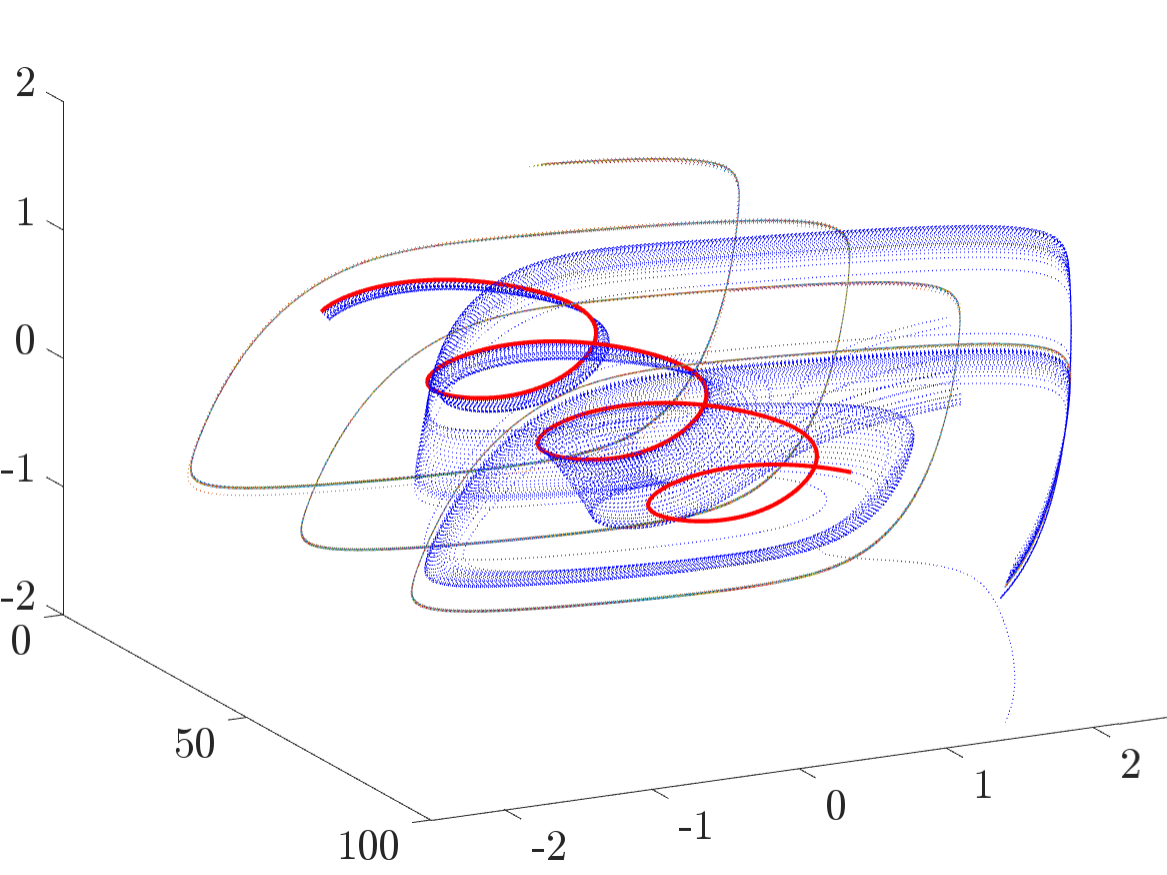} 
\put(42,72){$\delta=0.6$}
\put(15,6){$t$}
\put(70,1){$x$}
\put(-5,44){$y$}
\end{overpic}
\begin{overpic}[width=0.33\textwidth]{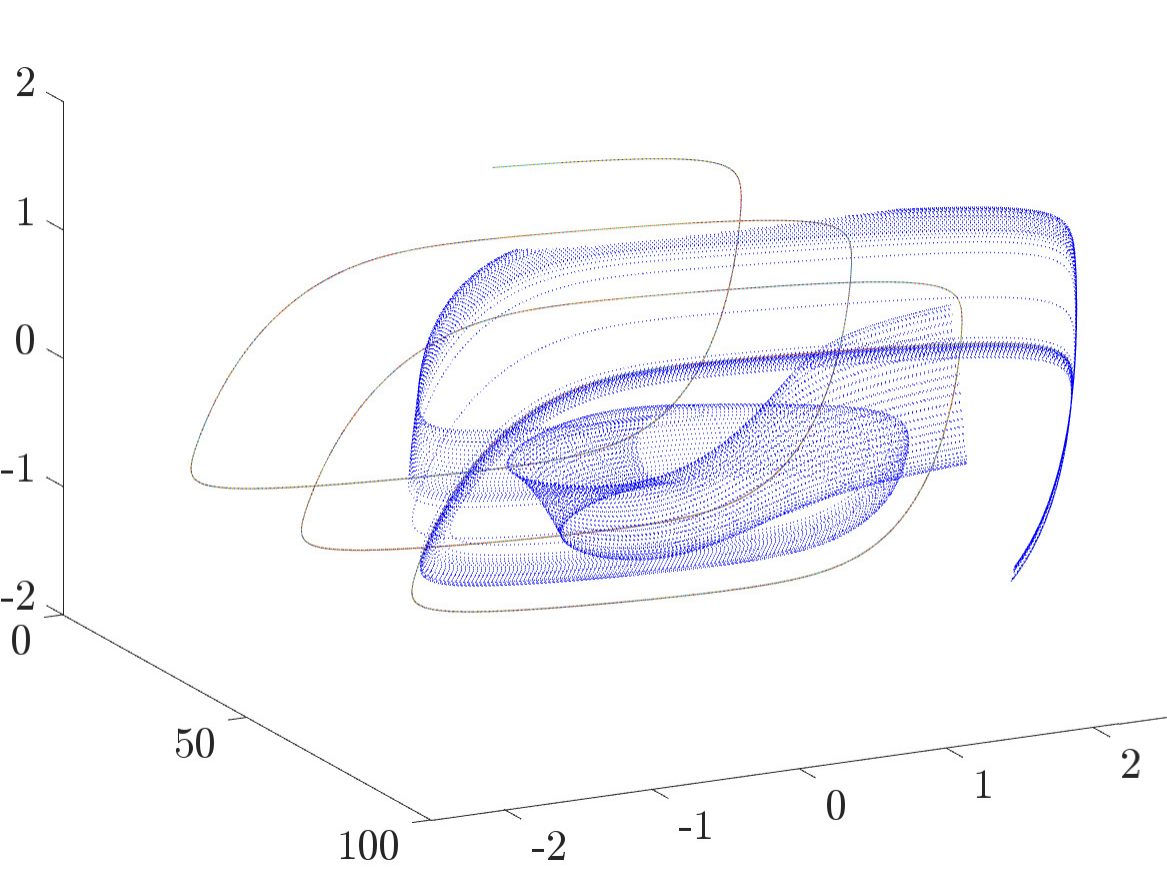} 
\put(42,72){$\delta=0.72$}
\put(15,6){$t$}
\put(70,1){$x$}
\put(-5,44){$y$}
\end{overpic}
\\[2ex]
\begin{overpic}[width=0.33\textwidth]{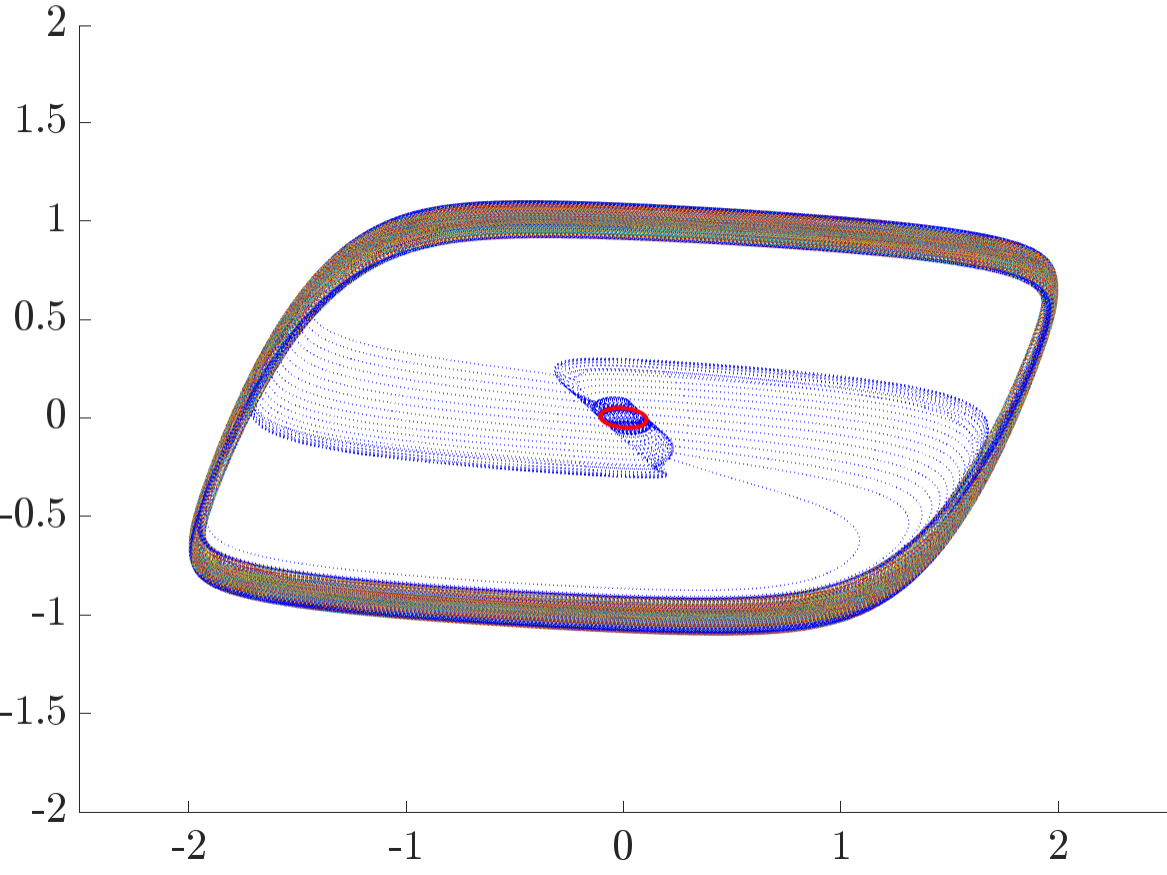} 
\put(51,-3){$x$}
\put(-1,39){$y$}
\end{overpic}
\begin{overpic}[width=0.33\textwidth]{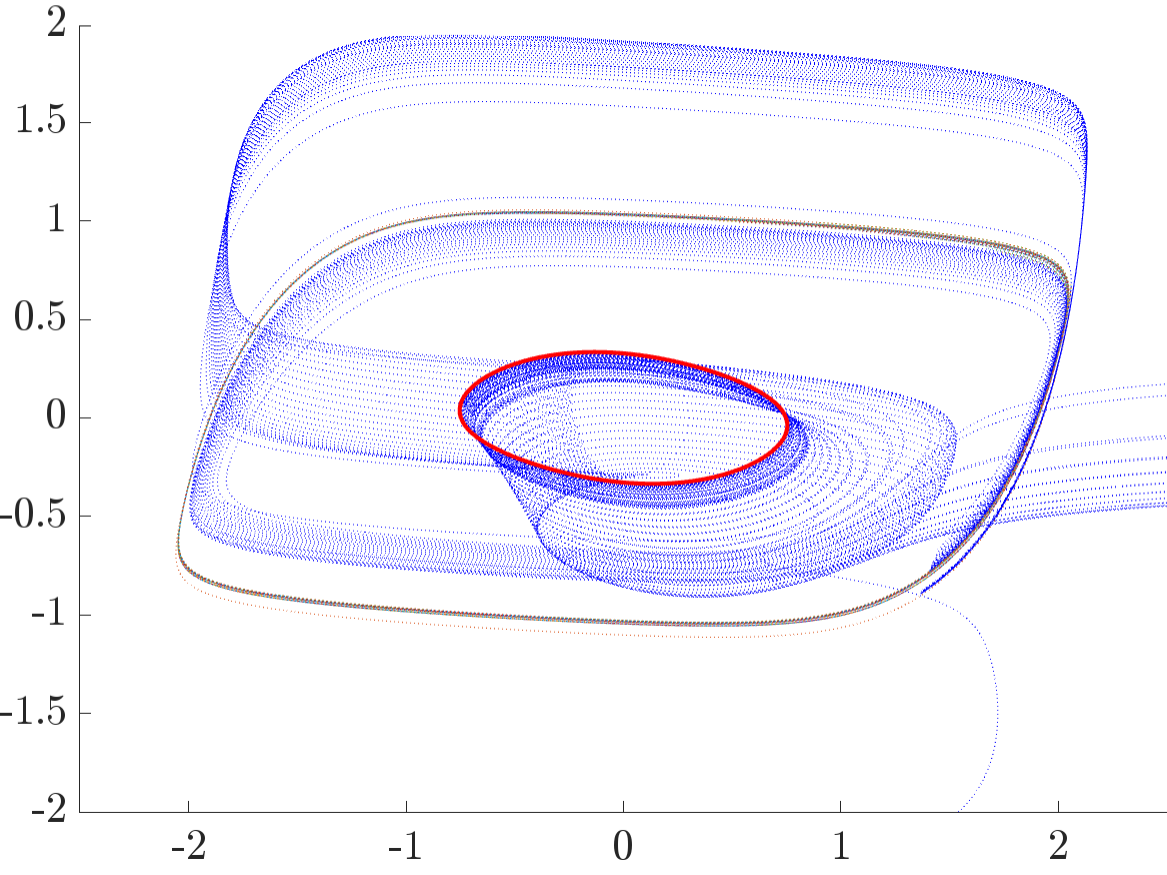} 
\put(51,-3){$x$}
\put(-1,39){$y$}
\end{overpic}
\begin{overpic}[width=0.33\textwidth]{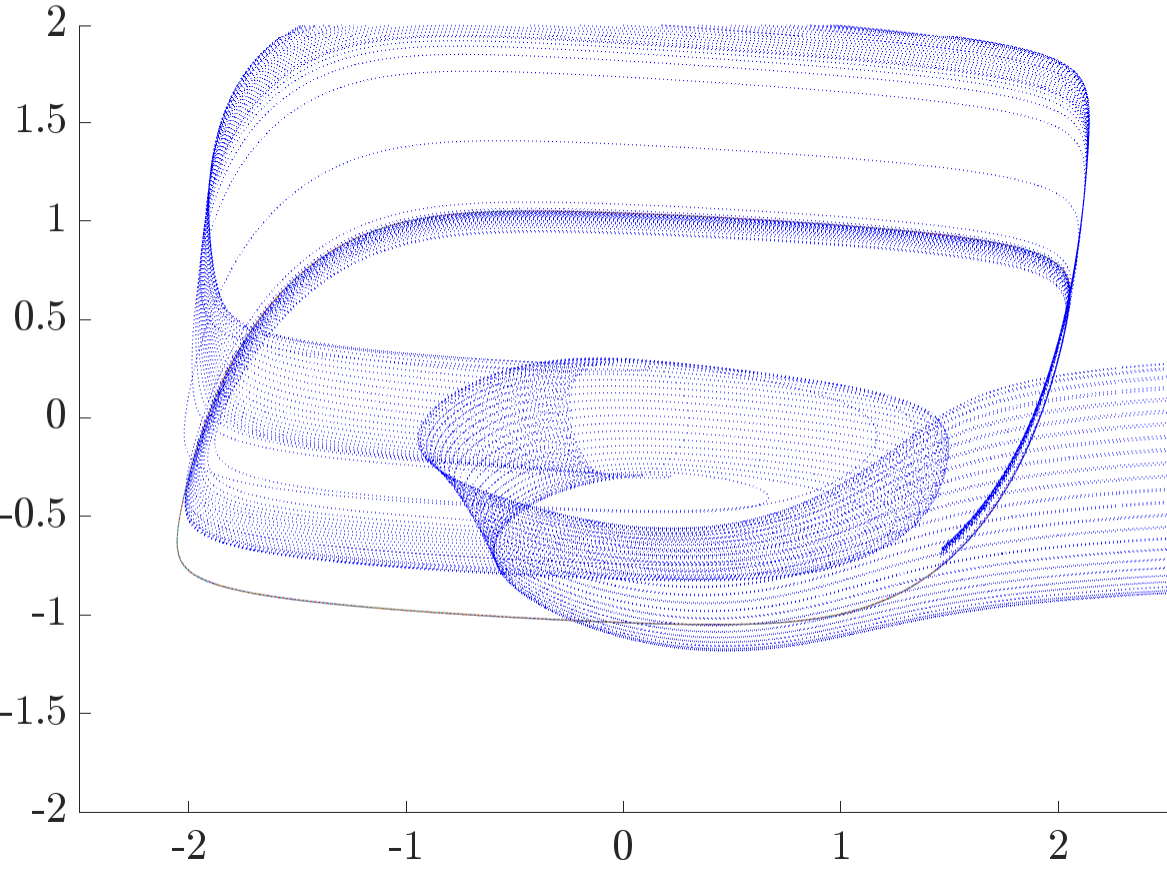} 
\put(51,-3){$x$}
\put(-1,39){$y$}
\end{overpic}
\\[2ex]
\begin{overpic}[width=0.33\textwidth]{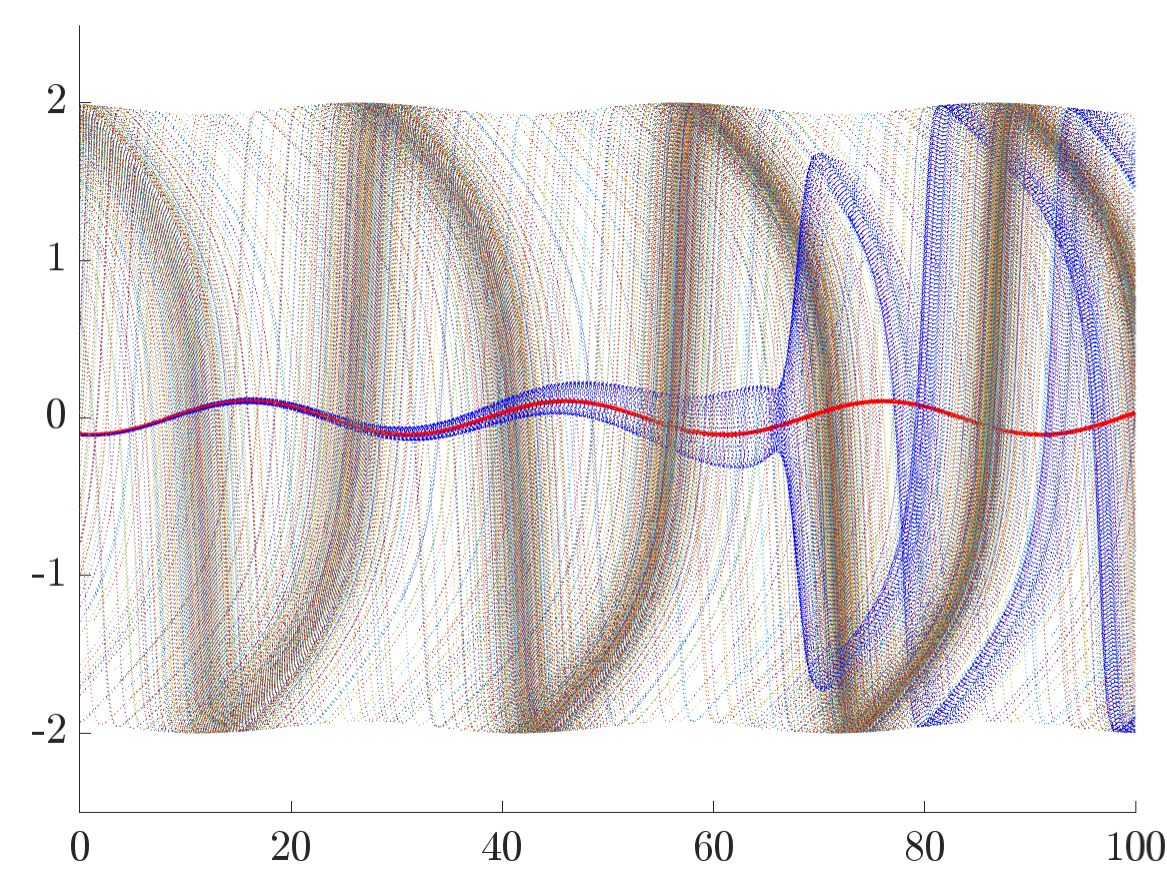} 
\put(51,-3){$t$}
\put(-1,39){$x$}
\end{overpic}
\begin{overpic}[width=0.33\textwidth]{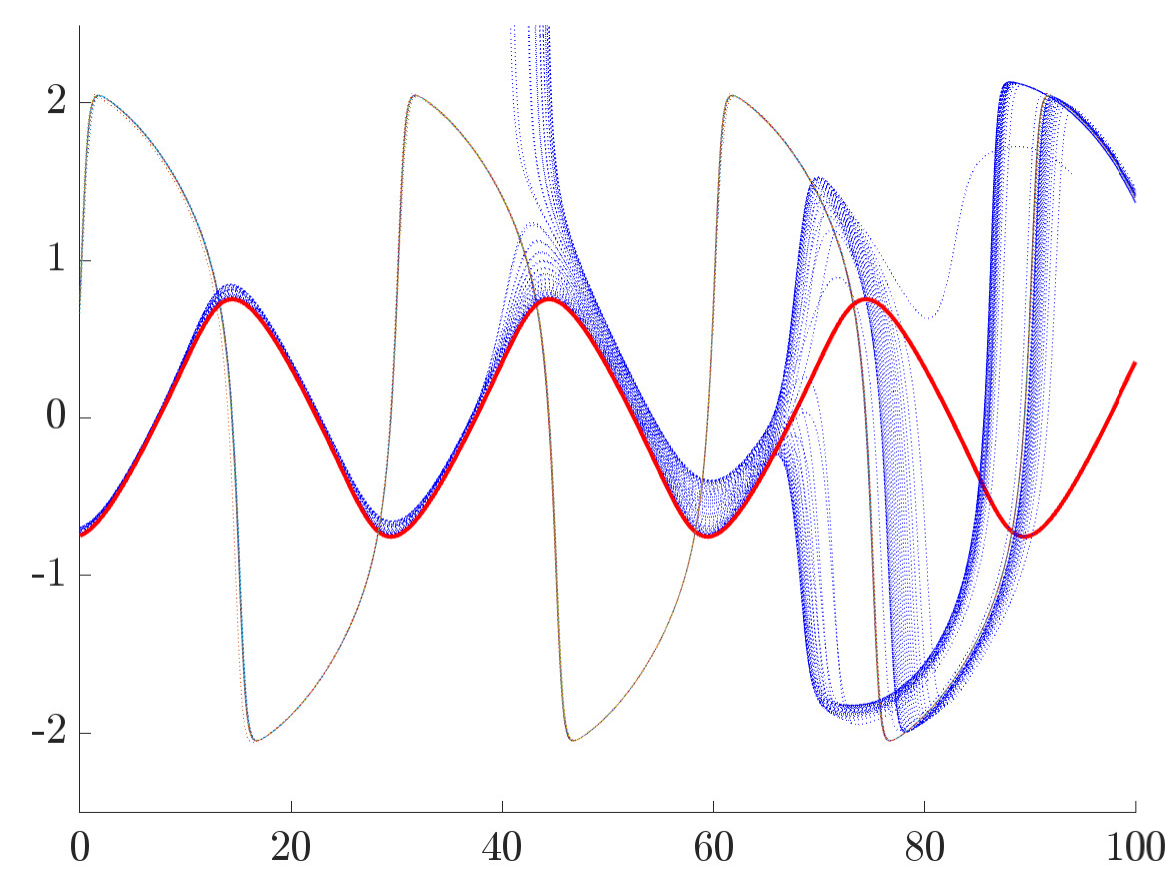} 
\put(51,-3){$t$}
\put(-1,39){$x$}
\end{overpic}
\begin{overpic}[width=0.33\textwidth]{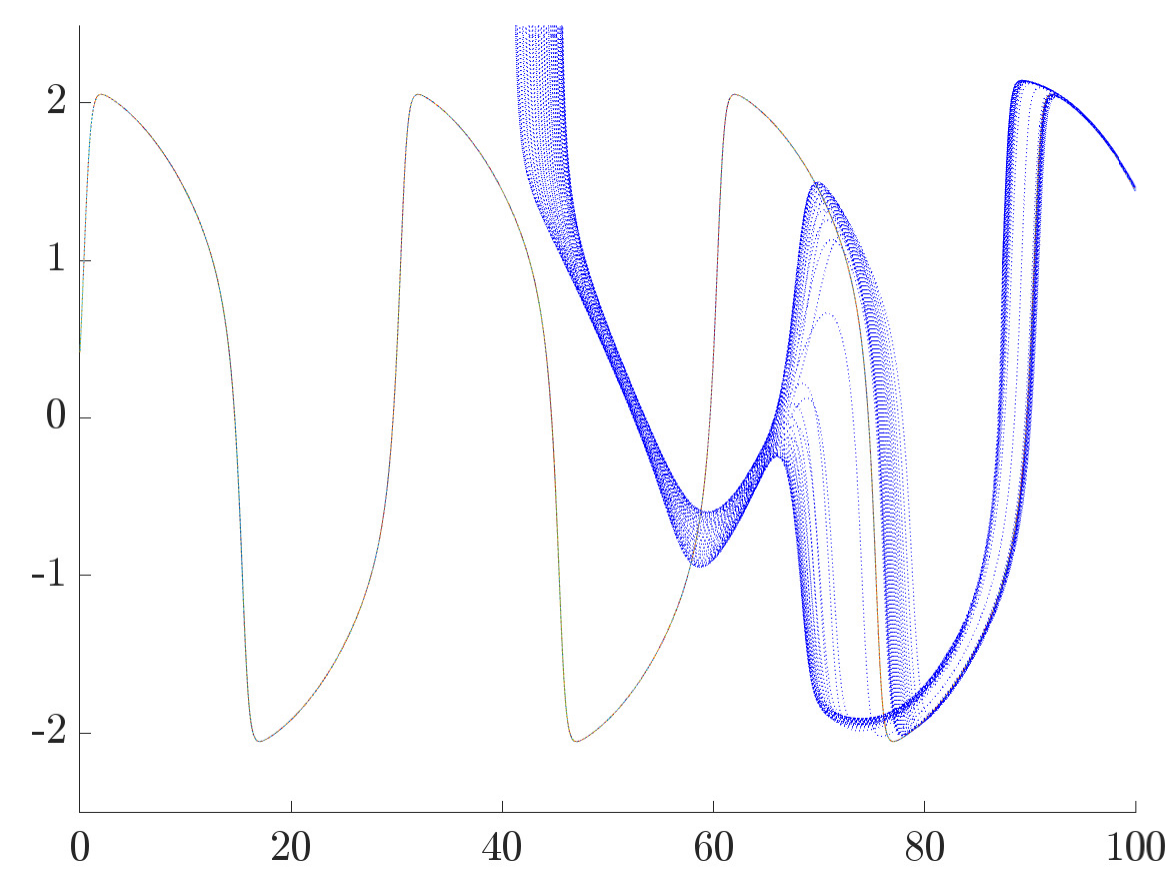} 
\put(51,-3){$t$}
\put(-1,39){$x$}
\end{overpic}
\caption{A phenomenology of the bifurcation through the appearance of unbounded solutions upon the variation of $\delta$. In this example, $v(t)=\cos(2\pi t/30)$, $a=0$, $b=0.4$, $\ep=0.1$. Five hundred solutions with initial conditions uniformly distributed in $[-200,50]$ are used to approximate the attracting invariant set of \eqref{eq:problem} (a transient of at least fifty time-steps is ignored and not plotted). The red trajectory represents the hyperbolic repelling solution of \eqref{eq:problem}. It is approximated numerically integrating backward in time twenty distinct initial conditions at time $t=20100$. The blue dotted lines are fifty solutions starting at time $t=66.7$ on the circle of radius $0.3$ centered at the origin and numerically integrated backward and forward in time. At $\delta=0.1$ these solutions are all bounded. At $\delta=0.6$, some of these solutions remain bounded whereas others blow up in finite time suggesting that the integral manifold does not exist anymore or it has been pierced. The hyperbolic repelling solution seems to persist, as the remaining bounded solutions converge to it in backward time. At $\delta=0.72$, all the considered solutions in blue blow up in finite time and the hyperbolic repelling solution does not seem to exist anymore.}
\label{fig:bif-periodic}
\end{figure*}

\begin{figure*}
    \centering
\begin{overpic}[width=0.33\textwidth]{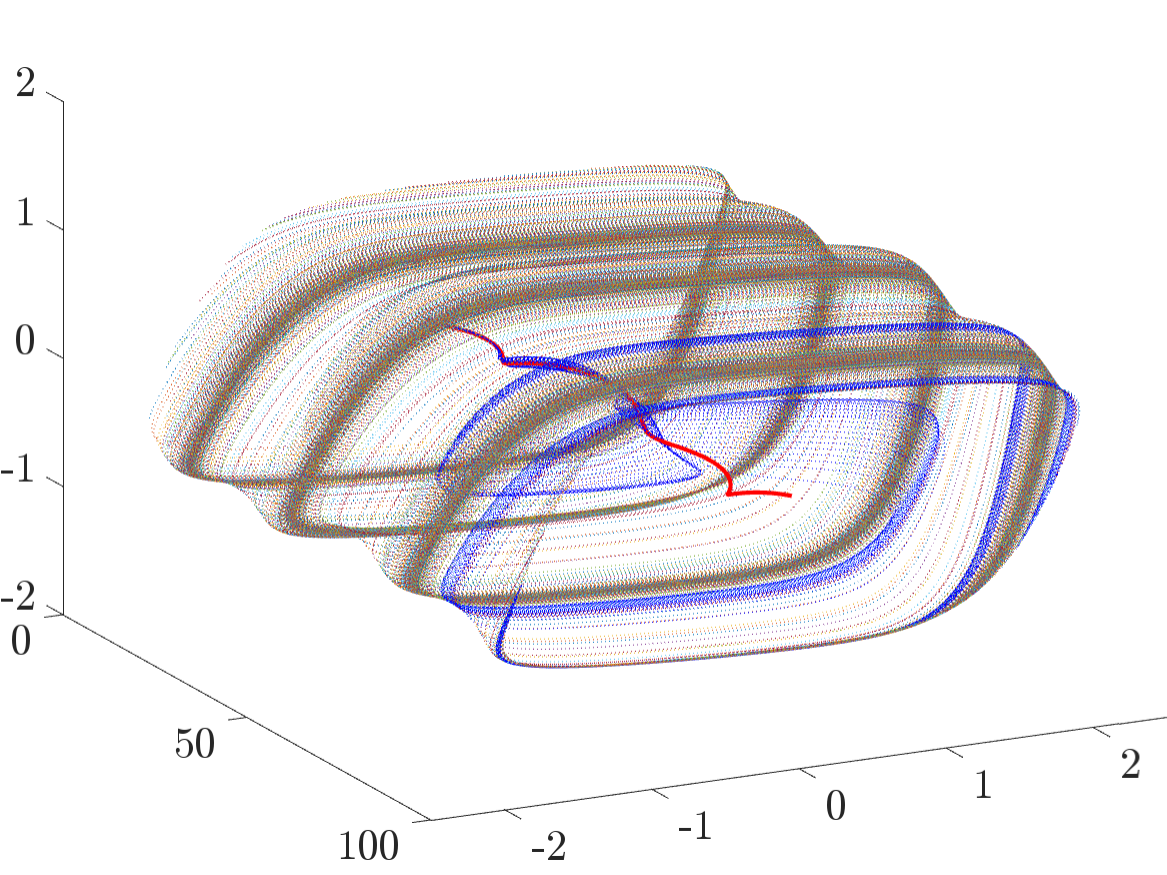} 
\put(42,72){$\delta=0.1$}
\put(15,6){$t$}
\put(70,1){$x$}
\put(-5,44){$y$}
\end{overpic}
\begin{overpic}[width=0.33\textwidth]{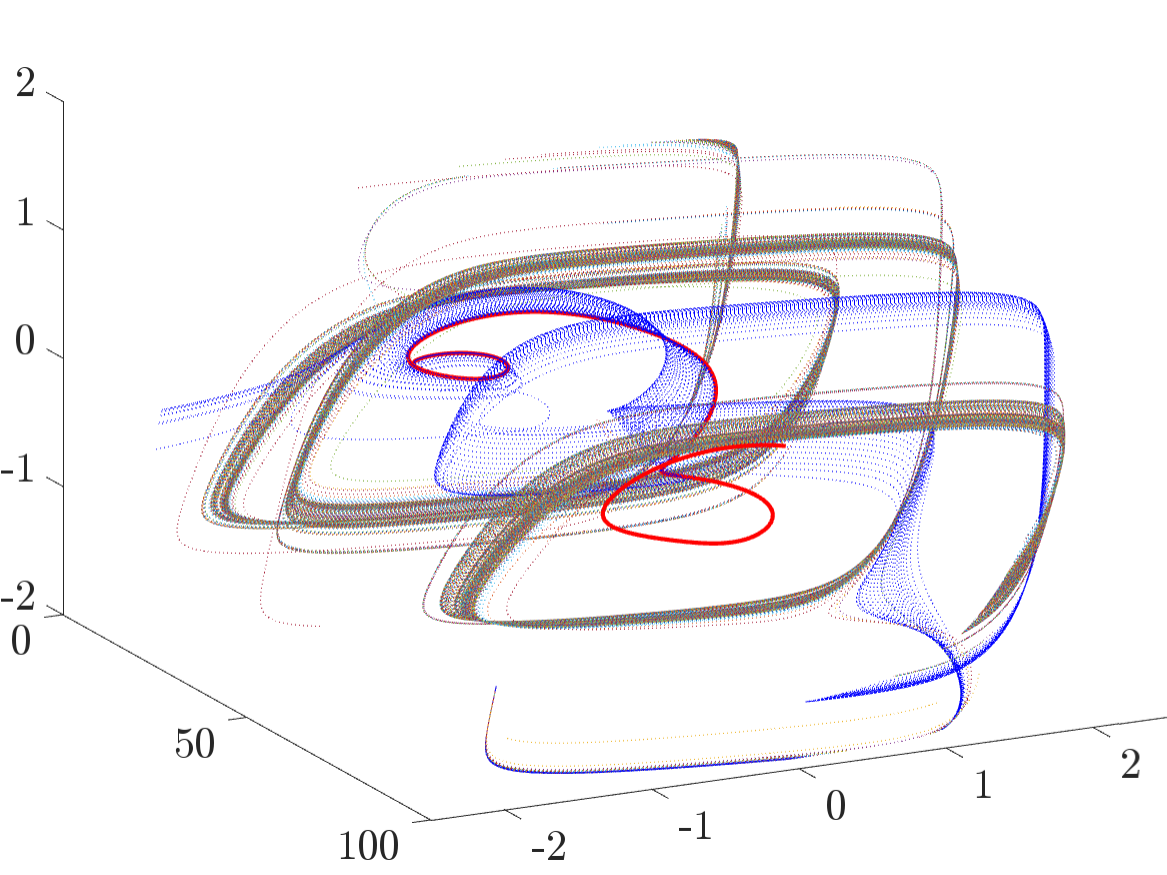} 
\put(42,72){$\delta=0.35$}
\put(15,6){$t$}
\put(70,1){$x$}
\put(-5,44){$y$}
\end{overpic}
\begin{overpic}[width=0.33\textwidth]{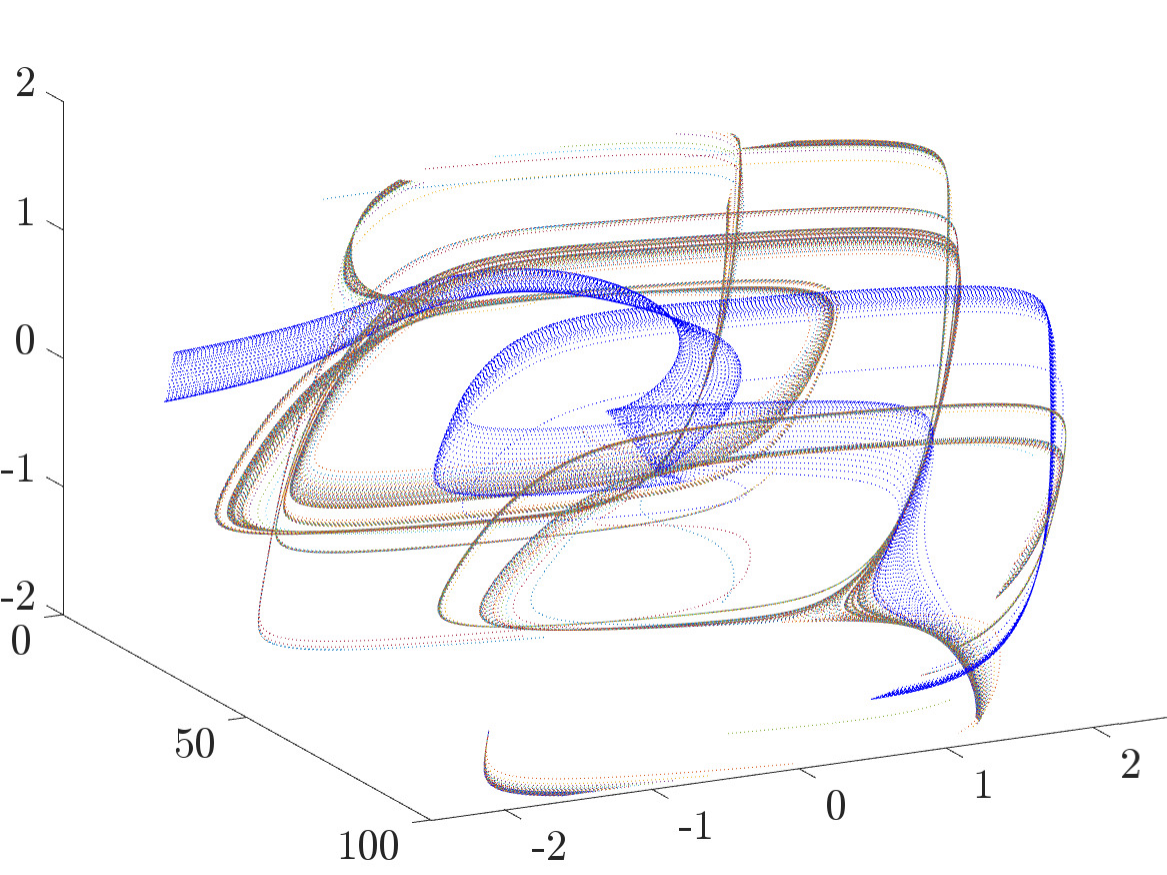} 
\put(42,72){$\delta=0.4$}
\put(15,6){$t$}
\put(70,1){$x$}
\put(-5,44){$y$}
\end{overpic}
\\[2ex]
\begin{overpic}[width=0.33\textwidth]{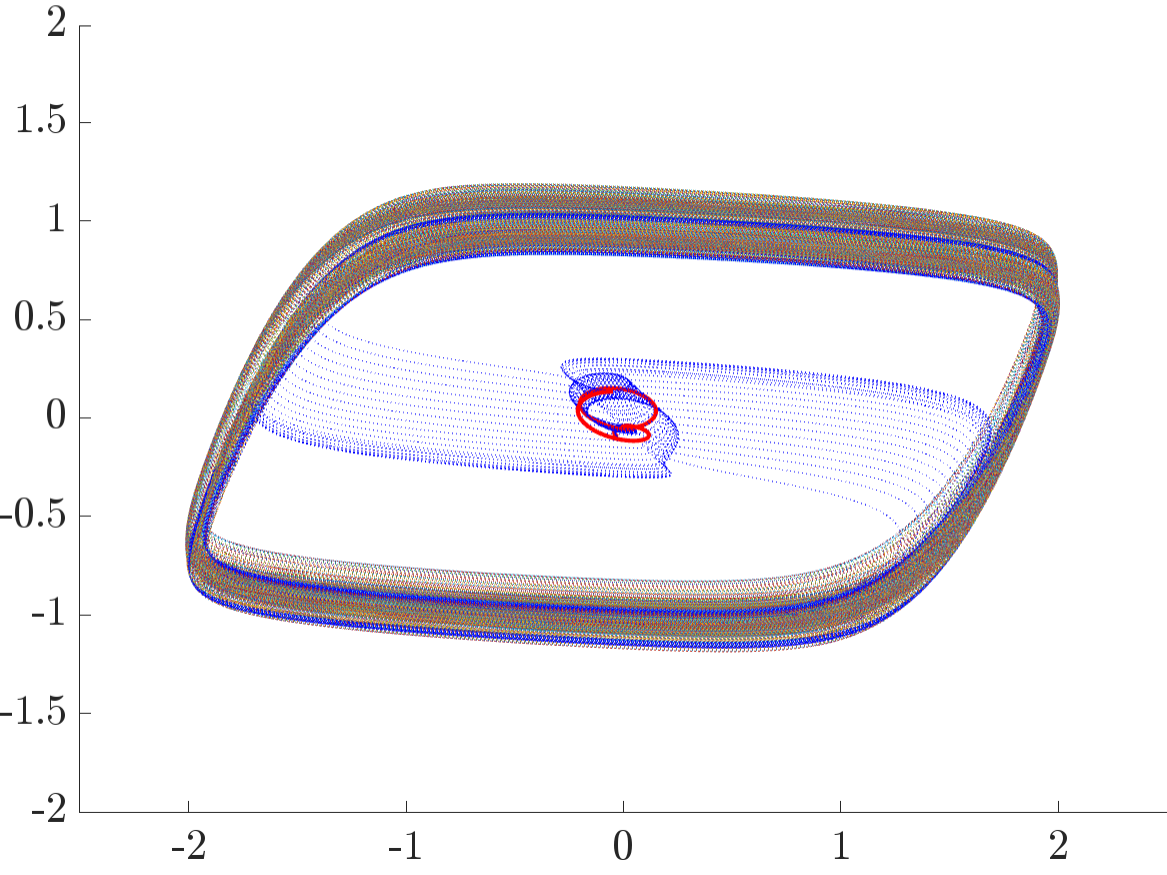} 
\put(51,-3){$x$}
\put(-1,39){$y$}
\end{overpic}
\begin{overpic}[width=0.33\textwidth]{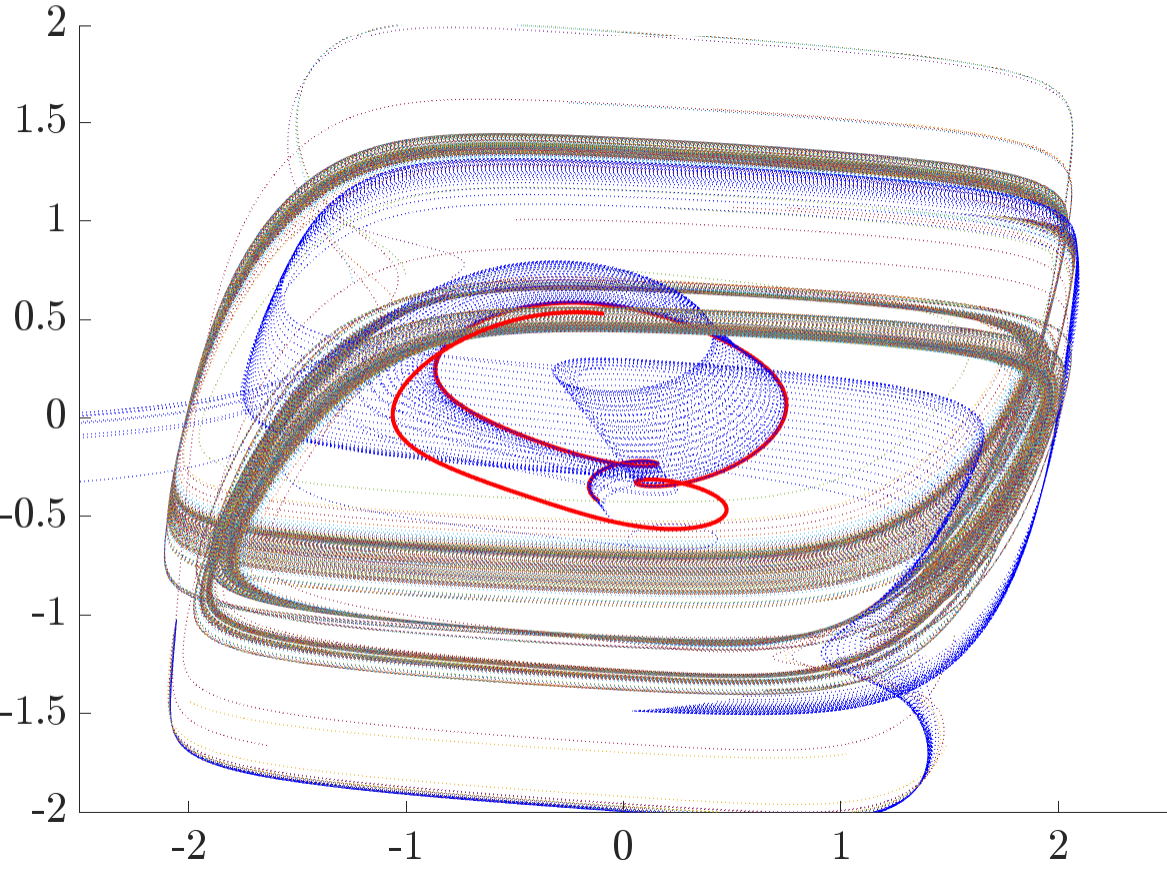} 
\put(51,-3){$x$}
\put(-1,39){$y$}
\end{overpic}
\begin{overpic}[width=0.33\textwidth]{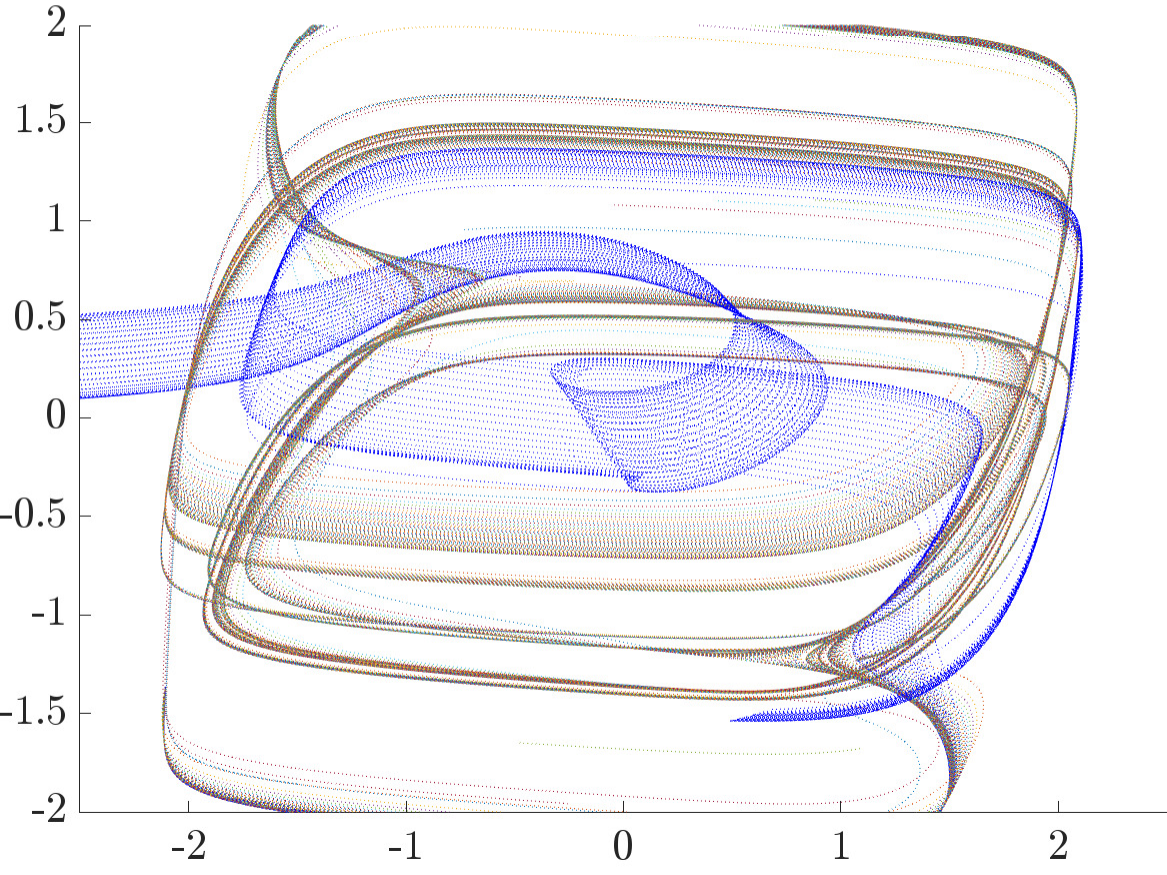} 
\put(51,-3){$x$}
\put(-1,39){$y$}
\end{overpic}
\\[2ex]
\begin{overpic}[width=0.33\textwidth]{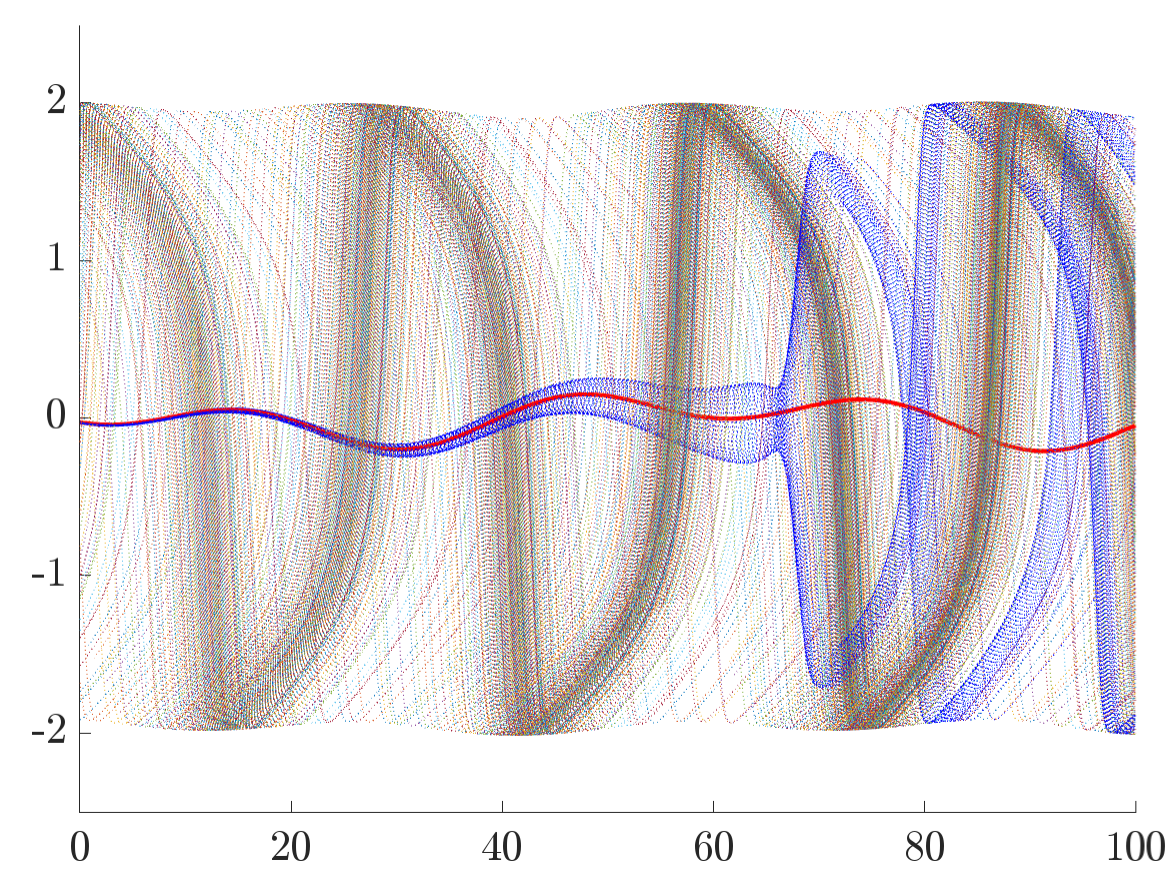} 
\put(51,-3){$t$}
\put(-1,39){$x$}
\end{overpic}
\begin{overpic}[width=0.33\textwidth]{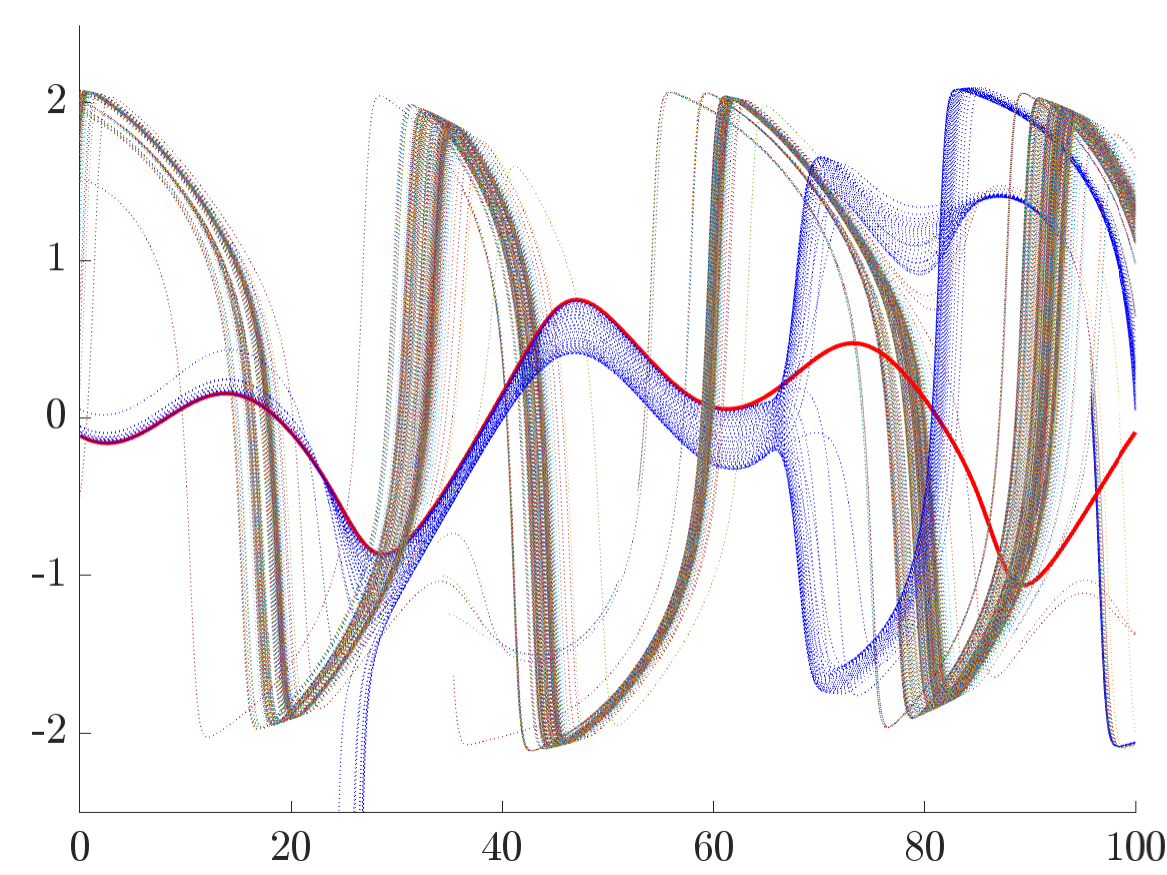} 
\put(51,-3){$t$}
\put(-1,39){$x$}
\end{overpic}
\begin{overpic}[width=0.33\textwidth]{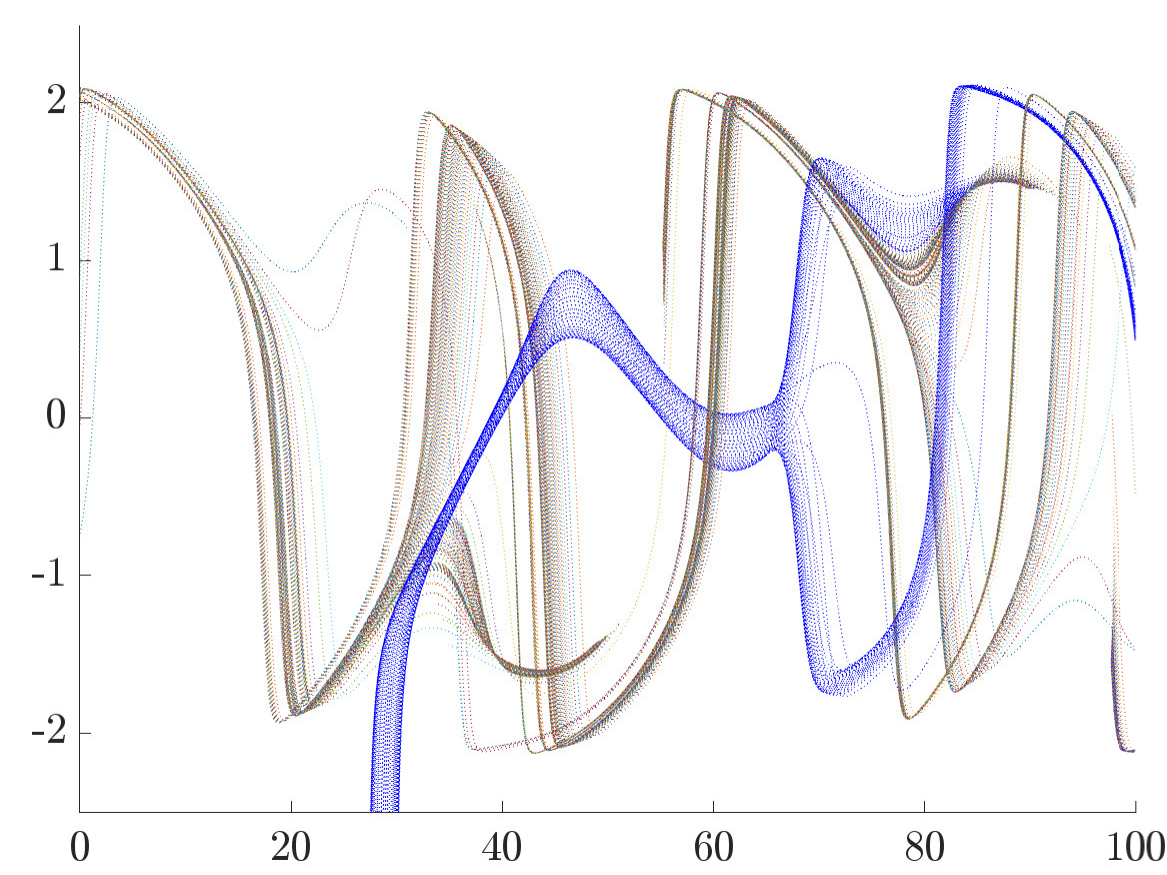} 
\put(51,-3){$t$}
\put(-1,39){$x$}
\end{overpic}
\caption{As for Figure \ref{fig:bif-periodic}, but for quasi-periodic forcing $v(t)=cos(2\pi t/30)+sin(2\pi t/(30\sqrt{5}))$ with  $a=0$, $b=0.4$, $\ep=0.1$. Five hundred solutions with initial conditions uniformly distributed in $[-200,50]$ are used to approximate the attracting invariant set of \eqref{eq:problem} (a transient of at least fifty time-steps is ignored and not plotted). The red trajectory represents the hyperbolic repelling solution of \eqref{eq:problem}. It is approximated numerically integrating backward in time twenty distinct initial conditions at time $t=20100$. The blue dotted lines are fifty solutions starting at time $t=66.7$ on the circle of radius $0.3$ centred at the origin and numerically integrated backward and forward in time. At $\delta=0.1$ these solutions are all bounded. At $\delta=0.35$, some of these solutions remain bounded whereas others blow up in finite time suggesting that the integral manifold does not exist anymore or it has been pierced. The hyperbolic repelling solution seems to persist, as the remaining bounded solutions converge to it in backward time. At $\delta=0.4$, all the considered solutions in blue blow up in finite time and the hyperbolic repelling solution does not seem to exist anymore.}
\label{fig:bif-quasi-periodic}
\end{figure*}

\section{A nonautonomous bifurcation} \label{sec:nonauton-bif}
In Section \ref{sec:global-attractor}, we have studied the nature of the global attractor of \eqref{eq:problem} for values of $\delta\in[0,1]$. In \ref{subsec:global-attractor-d=0} we have seen that, for values of $\delta$ sufficiently close to $\delta=0$, there exists a normally hyperbolic integral manifold enclosing the rest of bounded solutions, including a hyperbolic repelling solution for the problem. 
If $\delta$ takes values close to $\delta=1$, we showed in \ref{subsec:global-attract-d=1}  that there is a unique bounded solution and it is hyperbolic attracting and a copy of the base. 
For intermediate values of $\delta$ we have used a nonautonomous version of Tychonov's Theorem in \ref{subsec:global-attractor-int-d} to provide a local qualitative description of the global attractor that, although rigorous, it has practical limitation when a quantitative estimate is desired. 

It is clear that somewhere in the transition between the extreme values of $\delta$, the normally hyperbolic integral manifold breaks up, the hyperbolic completely unstable solution ceases to exist and a unique hyperbolic attracting solution arises. 
We can reasonably affirm that a bifurcation takes place. 

A natural candidate is a nonautonomous Hopf bifurcation pattern\cite{braaksma1987quasi,braaksma1990toward,johnson1994hopf,Franca2016nonautonomous,anagnostopoulou2015model,franca2019non,nunez2019li}.
The numerical evidence points at a two-step bifurcation\cite{arnold1995random,johnson2002two,anagnostopoulou2015model}. At first, the integral manifold collapses giving rise to a hyperbolic stable solution, and the hyperbolic repelling solution disappears only at a higher value of $\delta$. 
This is appreciable in Figures \ref{fig:bif-periodic} and \ref{fig:bif-quasi-periodic} for periodic and quasi-periodic forcing respectively. Note, in particular, that at small values of delta, initial conditions in a certain region of the phase space give rise to bounded solutions converging to the hyperbolic repelling trajectory in backward time. Upon varying the parameter some of these solutions start to blow up in finite time while others keep converging to the hyperbolic repelling solution in backward time; this fact points toward the conclusion that the integral manifold does not exists anymore. For even bigger values of the parameter, all the solutions cease to converge in backward time and rather blow up in finite time.

\begin{figure*}
    \centering
\begin{overpic}[width=0.33\textwidth]{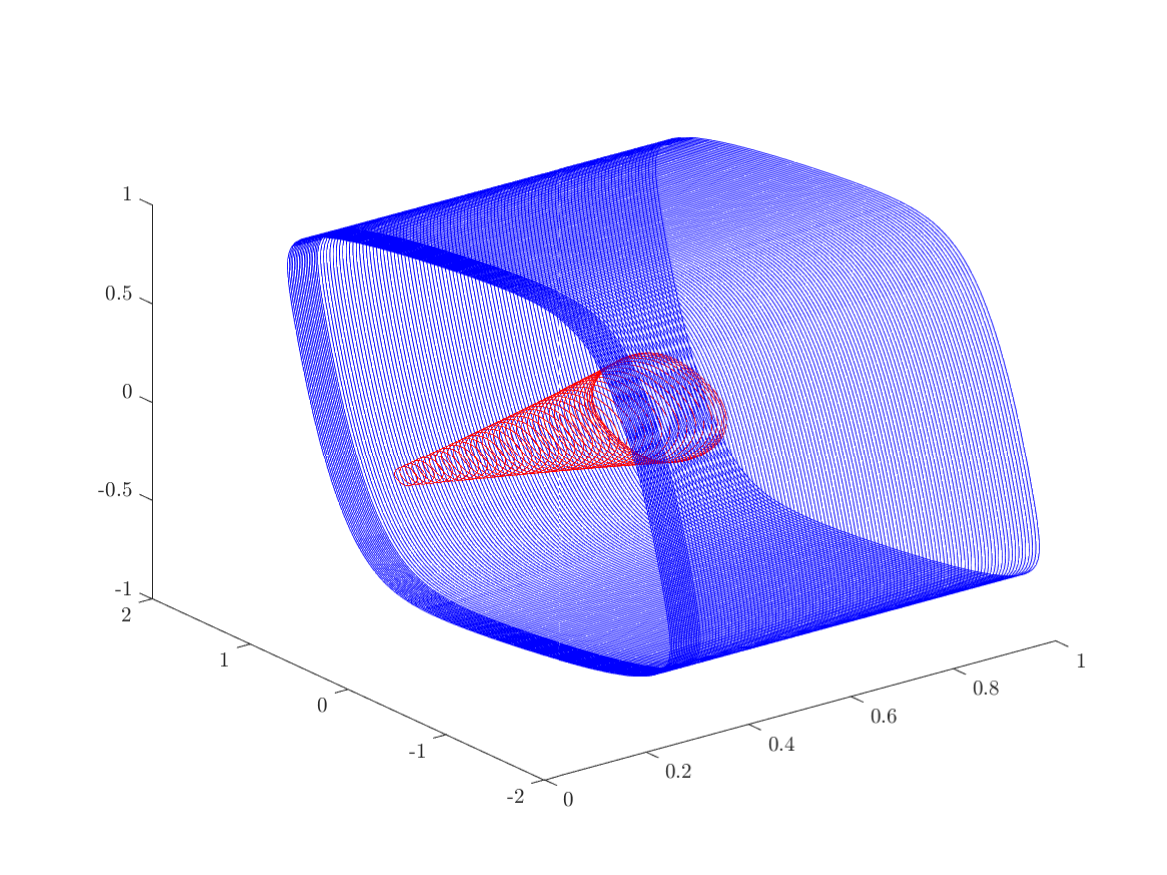} 
\put(20,0){$x$}
\put(80,0){$\delta$}
\put(-5,44){$y$}
\end{overpic}
\begin{overpic}[width=0.33\textwidth]{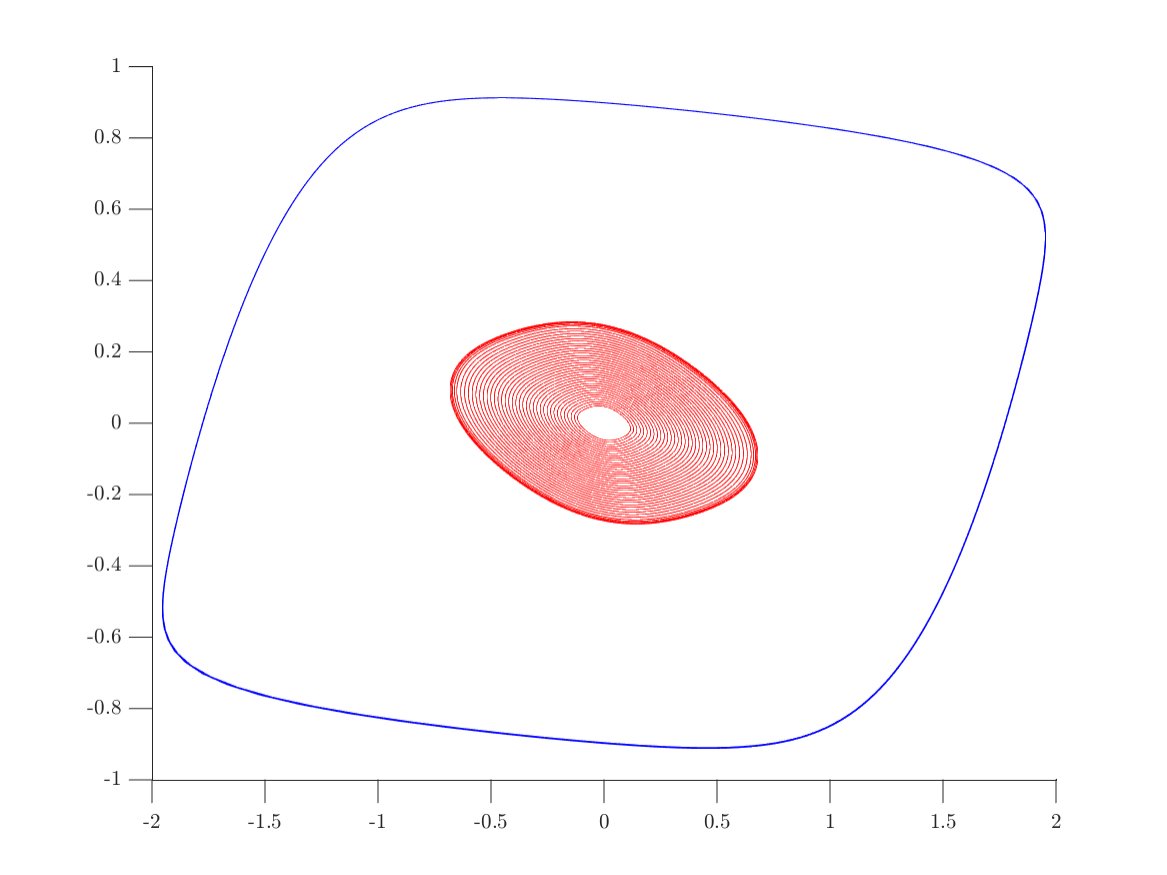} 
\put(51,-3){$x$}
\put(-2,39){$y$}
\end{overpic}
\begin{overpic}[width=0.33\textwidth]{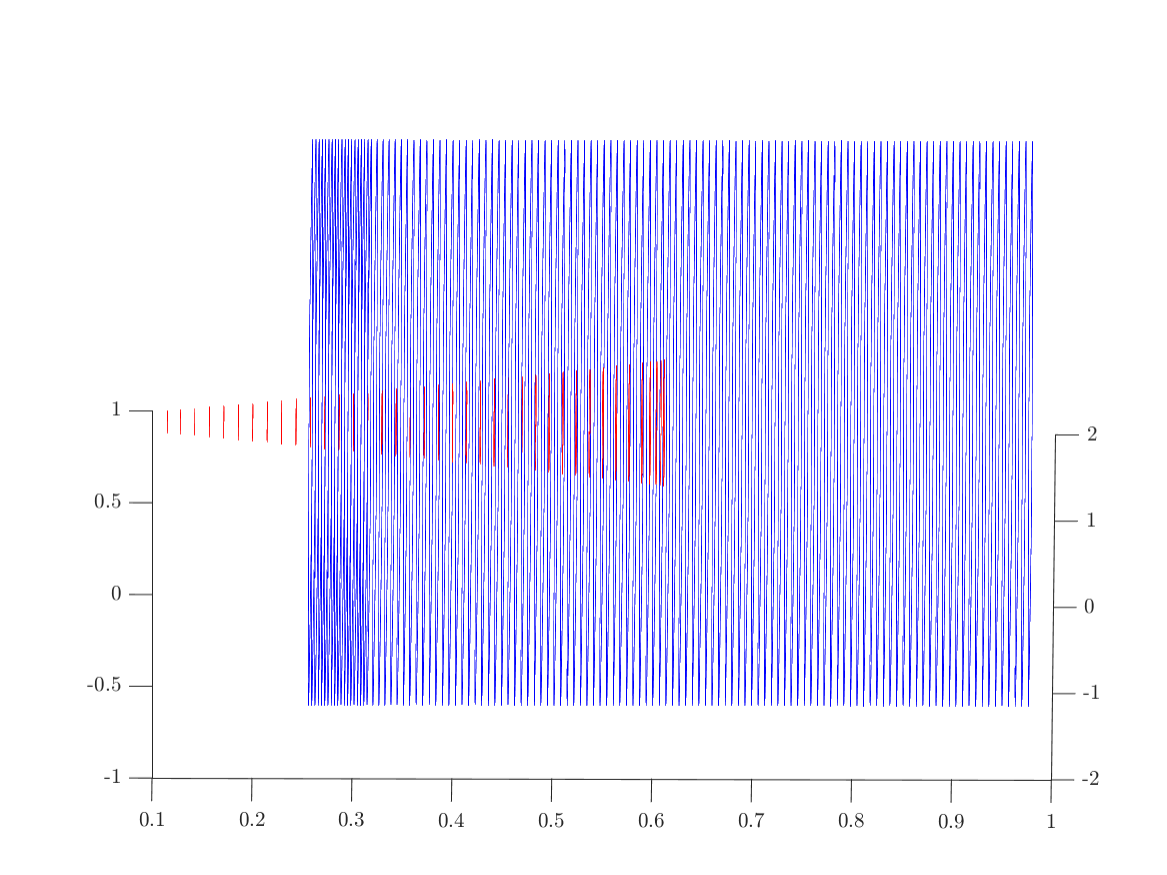} 
\put(50,-3){$\delta$}
\put(-3,40){$x$}
\end{overpic}
\caption{Continuation and numerical validation of hyperbolic periodic orbits of \eqref{eq:4D-FHN} with  $a=0$, $b=0.5$, $\ep=0.1$. The starting hyperbolic repelling solution is approximated numerically for $\delta=0.1$ and then continued forward. The hyperbolic attracting solution is approximated numerically for $\delta=0.9$ and then continued backward. The periodic orbits are here represented as projected on the $xy$-plane. It is possible to appreciate that the critical values of $\delta$ where they appear/disappear do not coincide.}
\label{fig:continuation-hyp-sols}
\end{figure*}

In order to acquire a better understanding of the phenomenon, we further explore the simpler case of periodic forcing. In particular, we  consider the special case where the $y$-component of the periodic attracting solution (with phase zero at time zero) of a twin FitzHugh-Nagumo system is fed into \eqref{eq:problem} as~$v$,
\begin{equation}\label{eq:4D-FHN}
\begin{split}
&\begin{sistema}
x' = (1-\delta)y+\delta v(t)-\frac{x^3}{3} +x\\
y' =\ep(a-x-by),\\
\end{sistema}\\
&\text{where } \big(u(t),v(t)\big)\text{ solves}\quad
\begin{sistema}
u' = v-\frac{u^3}{3} +u\\
v' =\ep(a-u-bv),
\end{sistema}\\
\end{split}
\end{equation}In this case, we are able to employ standard continuation techniques on the four dimensional autonomous problem \eqref{eq:4D-FHN}.

\subsection{Continuation of hyperbolic solutions}\label{subsec:continuation}

Let us first introduce some notation. Considering $\phi:\R\to \R$, $t\mapsto \phi(t)$, a periodic function of period $T$, we can write it in its Fourier expansion as 
$$
\phi(t) = \sum_{k=-\infty}^\infty c_ke^{\imag k \omega t}, \qquad c_k \in\mathbb{C}
$$
with $\omega = 2\pi T^{-1}$. Thanks to the periodicity in $t$, $c_k = \overline{c_{-k}}$. For a given $\nu$, we introduce the Banach space $\ell^1_\nu$ of complex sequences with the norm 
$$\| c \|_{\ell^1_\nu}:= \sum_{k=-\infty}^\infty |c_k|\nu^{|k|}.$$
If $\nu>1$, to have a bounded norm, the coefficients $|c_k|$ need to decrease geometrically away from $k = 0$. It is then numerically reasonable to consider the truncated Fourier sequence $ \hat c = [c_{-k},\dots, c_0, \dots, c_k]\in \mathbb{C}^{2k+1}$ and consider the embedding into the $\ell^1_\nu $ space. 

Let $y = (c_1, c_2, c_3, c_4) \in (\ell^1_\nu)^4$ be the Fourier representation of the solution to the ODE \eqref{eq:4D-FHN}, written as
\begin{equation}\label{eq:Fourier_problem}
\begin{split}
        &F(y) = \\ &=\begin{pmatrix}
        - \omega iKc_1 + c_2 - \frac{1}{3} c_1 \star c_1 \star c_1 + c_1
        \\ - \omega iKc_2 + \epsilon (a - c_1 - b c_2)
        \\
        -  \omega iKc_3 +(1-\delta) c_4 + \delta c_2 - \frac{1}{3} c_3 \star c_3 \star c_3 + c_3
        \\ - \omega iKc_4 + \epsilon (a - c_3 - b c_4)
    \end{pmatrix} = 0,
\end{split}
\end{equation}
with $(iKc)_k = i k c_k $ and $(c\star d)_k = \sum_{i+j=k} c_i d_j$. The latter is a finite sum if either $c$ or $d$ have finitely many elements.

Considering $\delta$ variable, the solution space for the Equation \eqref{eq:Fourier_problem} is $(\omega, \delta, y) \in \mathbb{R}\times\mathbb{R}\times (\ell^1_\nu)^4$, thus leaving us with two degrees of freedom, one associated with the parameter $\delta$, the other with rotations of periodic orbits, i.e. if $y$ is a solution so is $y(t+\tau)$ for any $\tau$. To fix this rotation, a new equation is introduced
$$
g_{\dot y _N}(y) = \langle y, \dot y_N \rangle + \beta = 0,
$$
where $\langle y, \dot y_N \rangle = \sum_{i=1}^4 \langle c_i, \dot c_i\rangle$ where $\langle c_i, \dot c_i\rangle = \sum_k (c_i)_k (\dot c_i)_k$ and $\dot y_N$ is a numerically computed kernel of the Frechet derivative of $F$. Then,  $\beta$ is defined as $\langle y_N, \dot y_N \rangle$ for $y_N$ a chosen numerical solution.

\begin{rmk}
Considering the 4-dimensional problem in Fourier space lifts the difficulties of time-matching the perturbing and the perturbed systems, i.e. $(u,v)$ and $(x,y)$.    
\end{rmk}

Once introduced the correct functional spaces and root-finding equations, the solution to $(F(\omega, \delta, y), g_{\dot y_N}(y)) = 0$ is, generically, a one-dimensional solution curve that we want to approximate by a sequence of points $\hat x_i$, starting from a given numerical approximation $\hat x_0 = (\omega_0, \delta_0, y_0)$. Let $v_1$ be a numerically approximated tangent to the solution curve, computed as the kernel of the Frechet derivative of $(F,g_{\dot y_N})$ at $\hat x_0$. Then an approximation for a new continuation point on the curve is $\tilde x _ 1 = \hat x _ 0 + \epsilon v_1$, where $\epsilon$ is an appropriately chosen step size. To refine this approximation, we apply a Newton algorithm of the zero-finding problem
$$
\begin{pmatrix}
    F(x), g(x), h(v,x)
\end{pmatrix}=0,
$$
where $h(v_1, x) = \langle v_1, x - \tilde x_1\rangle$ ensures not only unicity but also that the solution found by the Newton algorithm can be written as $ \tilde x _ 1 + p $, and the perturbation $p$ is perpendicular to the tangent, as presented in Figure \ref{fig:predictor_corrector}.

\begin{figure}
\includegraphics[width=0.35\textwidth, trim={3.5cm 5.5cm 5.5cm 0},clip]{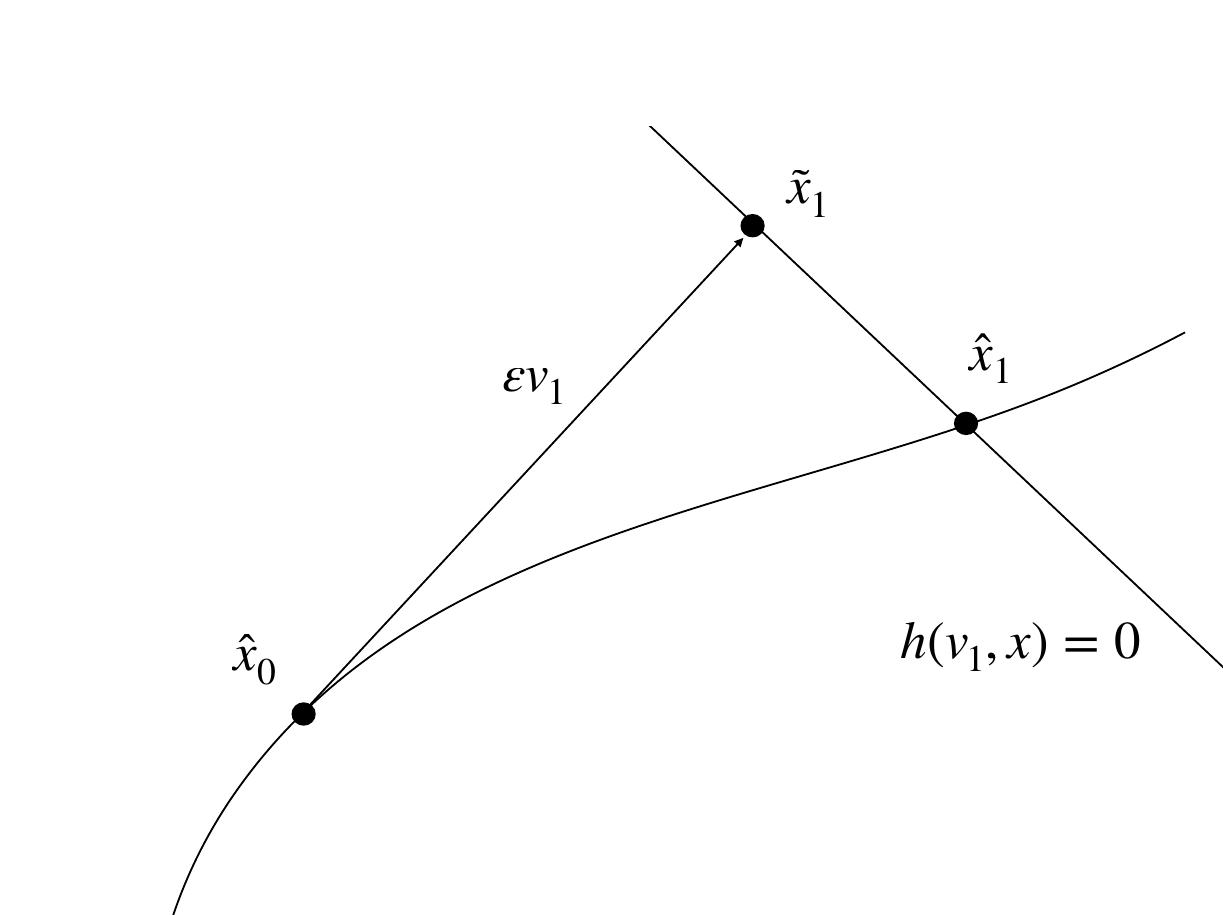}
\caption{A sketch of the predictor-corrector method.}
\label{fig:predictor_corrector}
\end{figure}

\begin{rmk}
The predictor-corrector algorithm is particularly well-suited for the continuation of hyperbolic solutions because it follows a branch independently of its stability.
\end{rmk}

Finally, the segment $x_s $ of analytical solutions connecting $x_0$ and $x_1$ and approximated by $x_0 + s(x_1 - x_0)$ solves the full zero-finding problem 
$$
G_s(x) = \begin{pmatrix}
    F(x)\\
    g_{\dot{y}_{Ns}}(x)\\
    h_s(x)
\end{pmatrix}
$$ 
where $s\in[0,1]$ and $h_s(x) = \langle v_s, x - \tilde x_s\rangle$ for coherently defined $v_s$ and $\tilde x_s$. 

Let $\hat x_0$ and $\hat x_1$ be two numerical solutions approximating $x_0$ and $x_1$. We here want to tackle the problem of proving the existence of an analytical branch of solutions in the neighborhood of the numerical approximations we computed. This is achieved here with the application of the \emph{radii polynomial approach}, based on the definition of a map $T_s(x) = x - A_s G_s(x)$, where $A_s$ is an approximation of the inverse of $DG_s(\hat x_s)$, the Frechet derivative of $G_s$. If $T_s$ is a contraction in a neighborhood of $\hat x_s$ for all $s\in [0,1]$, then by the Banach contraction theorem a unique fixed point $T_s(x_s) = x_s$ exists in that neighborhood. That is equivalent to saying that $A_sG_s(x_s) = 0$. If $A_s$ is non-singular, it then follows that $G_s$ has a unique solution $x_s$ in the same neighborhood. To make matters more precise, we sketch here the radii polynomial theorem\cite{radiipol1, radiipol2}.

\begin{thm}
    Let $T_s$ and $\hat x_s$ be as previously defined. Let us further define
    $$
    Y \geq \max_{s\in[0,1]} \| T_s(\hat x_s) - \hat x_s\|_{\sup}
    $$
    and
    $$
    Z(r) \geq \max_{s\in[0,1]}\max_{z,w \in B_1(0)} \| DT_s (\hat x_s + zr) wr \|_{\sup}
    $$
    where the norm is defined as
    $$\| x\|_{\sup} = \|(\omega, \delta, c_1, c_2, c_3, c_4)\| = \max\{ |\omega|, |\delta|, \max_{i = 1,2,3,4} \|c_i\|_{\ell^1_\nu}\}.$$
    Then if for a given $r^*$ it holds 
    $ Y + Z(r^*) < r^*$, $T_s$ is a contraction in the neighborhood $B_r(\hat x_s)$ for all $s$ and there exists a unique solution $x_s$ to $G_s(x_0) = 0$ in the same neighborhood.
\end{thm}

For the implementation, the BiValVe library \cite{church2023computer} in combination with the Inlab library \cite{rump1999intlab} has been used to continue and validate the periodic orbits presented. The complete code for the validation can be found on Git \cite{gitCode}. Figure \ref{fig:continuation-hyp-sols} represented the validated orbit computed. Thanks to the validation, we proved that two locally unique branches of solutions exist, one connecting the blue orbits and one the red ones. The validation of the blue orbit starts at $\delta = 0.9$ fails when reaching $\delta = 0.2563$ and for the red orbit starts at $\delta = 0.1$ and fails at $\delta= 0.6101$. Both branches have error bounds lower than $r^* < 2.8\star 10^{-4}$. This result also proves the existence of a interval of parameters $\delta$ where both orbits coexist. 

In order to complete the continuation analysis carried out in \ref{subsec:continuation}, we investigate the behavior of the Floquet exponents of the variational equation along the continued hyperbolic solutions. Particularly, the Floquet exponents are determined by approximating the Lyapunov spectra.
We recall that for the periodic case, and more in general for the almost periodic case, Lyapunov Exponents are well-defined thanks to the Birkhoff Ergodic Theorem and the fact that the flow on the extended phase space $\Hu(v)\times\R^2$ is uniquely ergodic.
We implemented the algorithm described in Section 2 of Ref.~\onlinecite{breda2024practical}. This method uses the QR factorization of a sequence of fundamental matrix solutions of the variational equation along the periodic orbits.
The estimation of the Lyapunov Exponents is achieved by truncating the limit for the calculation after 10 periods of the periodic solution. 
The results of the numerical simulation are shown in Figure \ref{fig:Lyapunov_spectrum}, where the stability of the blue orbit and the instability of the red one are confirmed.  

\begin{figure}
\begin{overpic}[width=0.45\textwidth, trim={.5cm 7.5cm 1.5cm 8.5cm},clip]{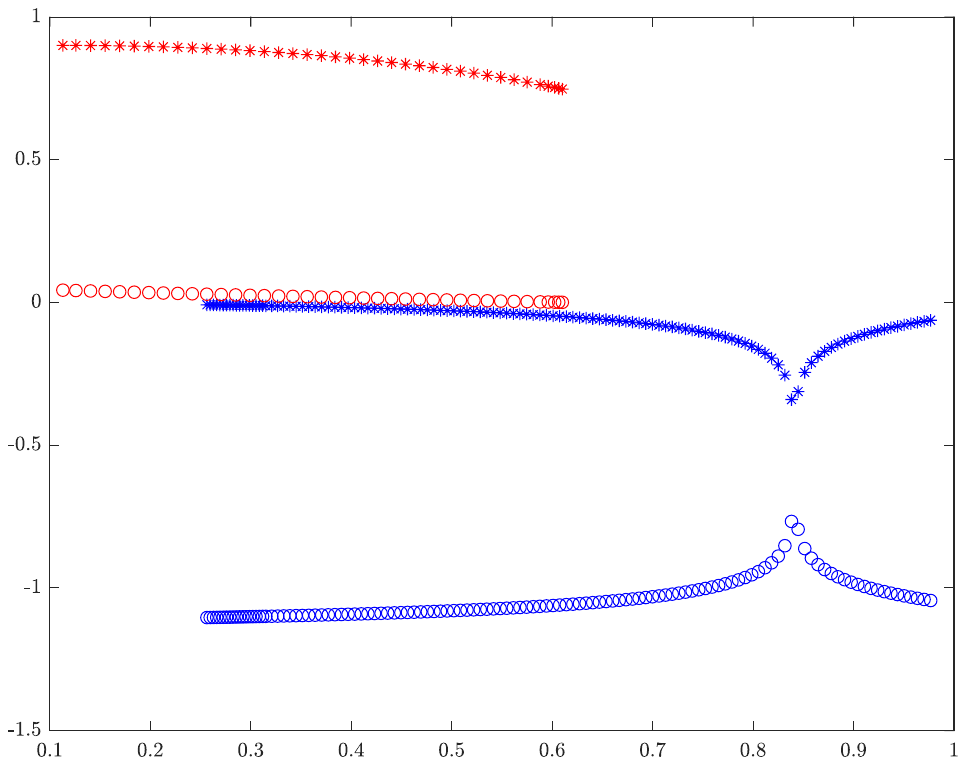} 
\put(50,0){$\delta$}
\put(5,30){\begin{sideways} spectrum \end{sideways}}
\end{overpic}
\caption{The Lyapunov spectra of the periodic solutions plotted w.r.t. $\delta$. In red the hyperbolically unstable solution, in blue the stable solution.}
\label{fig:Lyapunov_spectrum}
\end{figure}

It is also apparent that the two-step nature of the bifurcation cannot be resolved as a crossing of the zero for the spectrum at two different values of the parameters, since the simulation seems to suggest that the periodic hyperbolic orbits conserve their properties of stability till they cease to exist. In this sense, this phenomenon is different from the bifurcation discussed in Ref.~\onlinecite{johnson2002two}.

\section{Conclusions}
This work explores a natural and yet understudied framework: the parametric transition between autonomous and nonautonomous systems. We use the FitzHugh-Nagumo model as a paradigmatic example. This is a planar singularly perturbed system which has a periodic limit cycle and an unstable hyperbolic equilibrium in a suitable range of the parameter space. 
Moreover, the structure of the model allows to modify the equation for $\dot x$ by introducing a parametric linear interpolation between the $y$-component and a nonautonomous forcing $v(t)$ in a very simple fashion. An interesting case consists of taking $v$ as the $y$-component of a twin  FitzHugh-Nagumo model. In the extended phase space $\R\times\R^2$, the autonomous FitzHugh-Nagumo model has a global attractor with positive measure, whose boundary is an integral manifold. By its nature this attractor persists when $\delta>0$ is sufficiently small. We therefore obtain a nonautonomous system with an integral manifold enclosing the set of bounded solutions of the system, including a unique hyperbolic repelling solution.
On the other hand, if $\delta=1$ we obtain a nonautonomous skewed problem, made of a nonautonomous uncoupled cubic scalar equation and a linear inhomogenous scalar equation that can be explicitly solved once we have a solution to the first equation. Therefore, under suitable assumptions, we can analytically prove the existence of a unique globally attracting hyperbolic solution which persists for values of $\delta$ close to 1. 
This result is further supported by validated numerics techniques, allowing for a quantitative understanding of the solutions.
Hence, we analytically and numerically  studied the bifurcation phenomenon that leads to the collapse of the attracting integral manifold, the disappearance of the hyperbolic repelling solution and the birth of the unique hyperbolic globally attracting solution. 
The paper offers a rigorous technique to generate similar  bifurcation phenomena in nonautonomous problems derived by autonomous equations of Liénard type. 
A finer description of the bifurcation is however auspicable. Techniques from topological dynamics akin to the Diliberto map used in Ref.~\onlinecite{Franca2016nonautonomous} seem to be promising.
In this sense our work intends to stimulate the further investigation of this, similar and even more complicated scenarios using the transition between the autonomous and the nonautonomous realms as a basis. 
For example, we have hereby limited the presentation to a case where the FitzHugh-Nagumo model does not feature any particularly complicated dynamics. This is however not always the case as this differential problem is classically known to have several (autonomous) bifurcation points and a rather complicated unfolding. The case of higher dimensional systems and more complicated attractors also deserves attention. 

\section*{Acknowlegements}
The authors wish to deeply thank Courtney Quinn from the University of Tasmania for sharing her invaluable knowledge on the numerical approximation of Lyapunov Exponents. This allowed to rigorously complete the last part of Section \ref{sec:nonauton-bif}. \par\smallskip
IPL acknowledges partial support  by UKRI under the grant agreement EP/X027651/1, by MICIIN/FEDER project PID2021-125446NB-I00, by TUM International Graduate School of Science and Engineering (IGSSE) and by the University of Valladolid under project PIP-TCESC-2020.

EQ acknowledges full support by the DFG Walter Benjamin Programme QU 579/1-1.

CK acknowledges partial support  by a Lichtenberg Professorship of the VolkswagenStiftung and by the DFG Sachbeihilfe grant 444753754.

CK and IPL also acknowledge partial support of the EU within the TiPES project funded by the European Unions Horizon 2020 research and innovation programme under grant agreement No. 820970.

\section*{Author declarations} 

\subsection*{Conflict of Interest}
The authors have no conflicts to disclose.

\subsection*{Author Contributions}
All authors listed have made a substantial, direct, and intellectual contribution to the work and approved it for publication.

\section*{References}
\nocite{*}
\bibliography{References}

\end{document}